\documentclass[11pt,a4paper,reqno,tbtags]{amsart}

\usepackage[margin=1.1in]{geometry}
\usepackage{amsmath,amssymb,mathtools,graphicx}
\usepackage{url}
\usepackage{siunitx}
\usepackage[shortlabels]{enumitem}
\usepackage{tabularray}
\UseTblrLibrary{booktabs}

\newtheorem{theorem}{Theorem}[section]
\newtheorem{lemma}[theorem]{Lemma}
\newtheorem{corollary}[theorem]{Corollary}

\theoremstyle{remark}
\newtheorem{remark}[theorem]{Remark}

\newcommand{\D}{\mathbb{D}}
\newcommand{\N}{\mathbb{N}}
\newcommand{\Z}{\mathbb{Z}}
\newcommand{\Q}{\mathbb{Q}}
\newcommand{\R}{\mathbb{R}}
\newcommand{\J}{\mathcal{J}}
\newcommand{\dd}{\mathop{}\!{d}}
\newcommand{\A}{\mathbf{A}}

\DeclarePairedDelimiter{\abs}{\lvert}{\rvert}
\DeclarePairedDelimiter{\norm}{\lVert}{\rVert}
\DeclarePairedDelimiter{\innerproduct}{\langle}{\rangle}

\title
  [Monotonicity and the de Gennes bound]
  {Monotonicity and the de Gennes bound\\ for the magnetic Neumann Laplacian in the
  disk}

\author[C. Léna]{Corentin Léna}
\address
  [C. Léna]
  {University of Padua, Department of Management and Engineering -- DTG,
  Stradella S. Nicola 3, 36100 Vicenza, and Department of Mathematics
  ``Tullio Levi-Civita'', via Trieste 63, 35121 Padua, Italy}
\email{corentin.lena@unipd.it}

\author[M. Sundqvist]{Mikael Sundqvist}
\address
  [M. Sundqvist]
  {Department of Mathematics, Lund University, Sweden}
\email{mikael.persson\_sundqvist@math.lth.se}


\begin{document}

\begin{abstract}
We consider the lowest eigenvalue $\lambda(b)$ of the magnetic Neumann Laplacian
in the unit disk, for a constant magnetic field of strength $b>0$. We prove that
$\lambda$ is strictly increasing on $(0,+\infty)$. This means that strong
diamagnetism holds at every field strength, and not only at large ones. We also
show that the normalized energies at the successive crossings of angular-momentum
branches form a strictly increasing sequence; combined with the strong-field
asymptotics, this gives the global bound $\lambda(b)<\Theta_0 b$, where
$\Theta_0$ is the de Gennes constant. These results settle the three conjectures
formulated by Helffer and Léna for the disk. As a consequence, the local, or
spectral, critical field $H_{C_3}^{\mathrm{loc}}$ of Ginzburg--Landau theory is,
in the disk, uniquely determined for every value of the Ginzburg--Landau
parameter, and not only for large ones.

We also give a second proof of the bound $\lambda(b)<\Theta_0 b$, independent of
the first and of the results of Helffer and Léna, by a direct variational method:
trial states built from the de Gennes ground state for large fields, constant
trial states for small fields, and, on the remaining bounded field interval,
finite-dimensional spaces of polynomial trial states certified by finitely many
exact computations in rational arithmetic. That proof uses no asymptotic input.
It yields in addition an explicit upper bound for $\lambda(b)$, valid above an
explicit field strength, whose two leading terms are those of the strong-field
asymptotics.
\end{abstract}

\maketitle

\section{Introduction and main results}

\subsection{The magnetic Neumann Laplacian in the disk}

Let
\[
 \D=\{\mathbf{x}=(x_1,x_2)\in\R^2\colon x_1^2+x_2^2<1\},
 \qquad
 \A(x_1,x_2)=\frac12(-x_2,x_1).
\]
Thus $\operatorname{curl}\A=1$. For $b\in\R$, let $H_b$ be the self-adjoint
operator in $L^2(\D)$ associated with the closed quadratic form
\[
 q_b[u]=\int_{\D}\abs{(-i\nabla+b\A)u}^2\dd \mathbf{x},
 \qquad u\in H^1(\D).
\]
Equivalently, $H_b=(-i\nabla+b\A)^2$ with magnetic Neumann boundary condition
\[
 \mathbf{n}\cdot(-i\nabla+b\A)u=0
 \quad\text{on }\partial\D,
\]
where $\mathbf{n}$ is the outward unit normal. Since $\A\cdot\mathbf{n}=0$ on the
circle, this boundary condition is simply $\partial_{\mathbf{n}} u=0$. We write
\[
 \lambda(b)=\inf\operatorname{spec}H_b
\]
for the lowest eigenvalue, also called the \emph{ground-state energy}. It can be
shown that the domain $\mathcal D(H_b)$ of $H_b$ is stable under conjugation by
the operator $\Gamma$, defined by $\Gamma u:=\overline{u}$, and that
$H_{-b} \Gamma=\Gamma H_b$. Since $\Gamma$ is anti-unitary, it follows that
$H_{-b}$ and $H_b$ have the same spectrum, and in particular
$\lambda(-b)=\lambda(b)$. We can therefore restrict ourselves to the case $b\ge0$
without loss of generality.

The de Gennes constant is defined as
\begin{equation}\label{eq:Theta0}
 \Theta_0=\inf_{\xi\in\R}\mu(\xi),
\end{equation}
where $\mu(\xi)$ denotes the lowest eigenvalue of
$-d^2/dt^2+(t-\xi)^2$ in $L^2(\R_+)$ with a Neumann condition at the
origin. The strong-field asymptotics for smooth planar domains imply, in the
present case,
\begin{equation}\label{eq:de-gennes-limit}
 \lim_{b\to\infty}\frac{\lambda(b)}b=\Theta_0;
\end{equation}
see~\cite[Proposition~8.3.1]{MR2662319}.

Monotonicity of a magnetic Neumann ground-state energy with respect to the field
strength is usually referred to as strong diamagnetism. It is not a consequence
of the standard diamagnetic inequality, and it is sensitive to the topology of
the domain: in a multiply connected domain, flux effects of Little--Parks type
can produce oscillations, and hence destroy global monotonicity~\cite{MR3324161}.
Fournais and Helffer proved eventual strong diamagnetism for smooth bounded
planar domains~\cite{MR2394546}; see also~\cite{MR2662319} for its role in
surface superconductivity. For the disk, Helffer and Léna obtained a complete
description of the crossings of consecutive angular-momentum branches and
identified the branch that realizes the ground state~\cite{MR4947381} (see
Section~\ref{subsec:branches} for more details). They formulated three
conjectures: global monotonicity of $b\mapsto\lambda(b)$, strict increase of the
normalized energies at the successive crossings, and the bound
$\lambda(b)<\Theta_0b$ for every $b>0$. All three are proved below. The branches
and their crossings had previously been studied by the physicist Saint-James~\cite{SaintJames}.
We recomputed his graphs with the help of Mathematica, as
shown in Figure~\ref{fig:eigenvalue}.

\begin{figure}[htb]
  \centering
  \includegraphics[width=.75\textwidth,page=1]{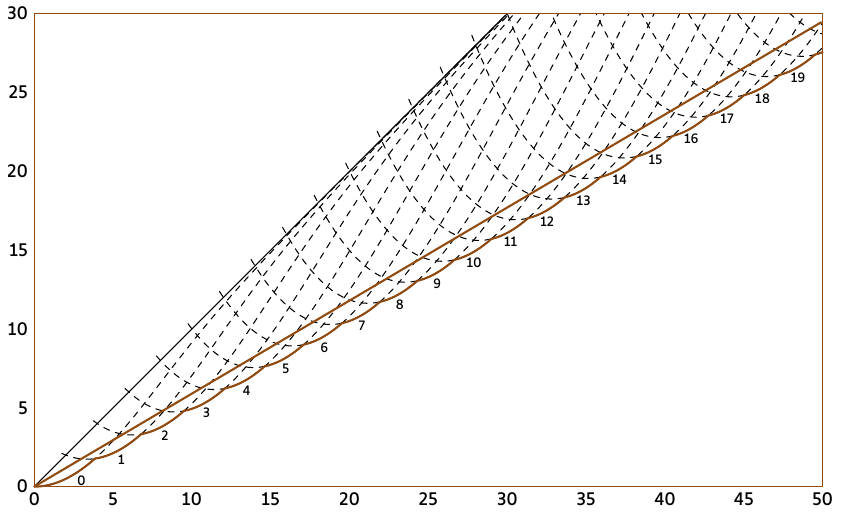}
  \caption{
  The eigenvalue curves \(b \mapsto \lambda _ n (b)\) for \(0\leq n\leq 19\). The
  smallest eigenvalue \(\lambda(b)\) and the line \(b \mapsto \Theta _ 0 b\) are
  drawn with thicker lines.}
  \label{fig:eigenvalue}
\end{figure}

\subsection{Angular-momentum branches}\label{subsec:branches}

Since the magnetic potential is tangential on $\partial\D$, the magnetic Neumann
condition reduces in polar coordinates to the ordinary radial Neumann condition.
We can write the orthogonal decomposition as follows:
\begin{equation}
	\label{eq:decomposition}
	L^2(\D)=\bigoplus_{n\in \Z}\mathcal H_n,
\end{equation}
where, using the polar coordinates $(r,\theta)$ in $\D$,
\begin{equation*}
	\mathcal H_n=\left\{ f(r)e^{-in\theta}/\sqrt{2\pi}\colon f \in L^2((0,1),r\dd r)\right\}.
\end{equation*}
This is simply a reformulation of the Fourier decomposition with respect to the
variable $\theta$. The integer parameter $n$ can be interpreted as the angular
momentum. We accordingly call the space $\mathcal H_n$ the \emph{$n$-th angular
momentum sector}. One can check that the subspace $\mathcal H_n\cap \mathcal
D(H_b)$ is stable under $H_b$. Using the obvious identification of $\mathcal H_n$
with $L^2((0,1),r\dd r)$, the restriction of $H_b$ to $\mathcal H_n$ can be
identified with the operator
\begin{equation}
 H_{n,b}
 =-\frac{d^2}{dr^2}-\frac1r\frac d{dr}
  +\left(\frac nr-\frac b2r\right)^2,
\end{equation}
acting in $L^2((0,1),r\dd r)$, with $f'(1)=0$. From this and from the
decomposition~\eqref{eq:decomposition}, it follows that the spectrum
$\sigma(H_b)$ of $H_b$ can be written
\begin{equation*}
	\sigma(H_b)=\bigcup_{n\in \Z}\sigma(H_{n,b}),
\end{equation*}
where the union takes into account the multiplicities. In particular, if we
denote the lowest eigenvalue of $H_{n,b}$ by $\lambda_n(b)$,
\[
 \lambda(b)=\min_{n\in\Z}\lambda_n(b).
\]
The operator $H_{n,b}$ can also be defined through its quadratic form
\begin{equation}\label{eq:qnb}
 q_{n,b}[f]=\int_0^1\biggl[
   \abs{f'(r)}^2+\left(\frac nr-\frac b2r\right)^2\abs{f(r)}^2
  \biggr]r\dd r,
\end{equation}
whose domain consists of those $f\in L^2((0,1),r\dd r)$ for which the right-hand
side of~\eqref{eq:qnb} is finite; this implies a regularity condition at the origin.
For $n\geq1$ one has $H_{-n,b}=H_{n,b}+2nb$, so a negative angular-momentum $n$
cannot realize the ground-state energy when $b>0$. It is therefore enough to
consider $n\geq 0$, so that
\[
 \lambda(b)=\min_{n\ge0}\lambda_n(b).
\]
The eigenvalue $\lambda_n(b)$ is simple, being the ground-state energy of a
one-dimensional Sturm--Liouville operator; consequently $b\mapsto\lambda_n(b)$ is
real-analytic on $(0,+\infty)$ by analytic perturbation
theory~\cite[Chapter~VII]{MR407617}. We call this function the \emph{$n$-th
angular momentum branch}. It also follows from Sturm--Liouville theory that an
eigenfunction of $H_{n,b}$ associated with $\lambda_n(b)$ does not vanish in
$(0,1)$. For future reference, we denote by $f_{n,b}$ the corresponding positive
eigenfunction, normalized by $\int_0^1 f_{n,b}(r)^2r\dd r = 1$.

We shall use the following two results of Helffer and Léna~\cite{MR4947381}.

\begin{theorem}[{Helffer--Léna~\cite[Theorem~1.2 and Proposition~3.3]{MR4947381}}]\label{thm:HL}
\leavevmode
\begin{enumerate}
 \item[\textup{(HL1)}]
  There is a strictly increasing sequence $(\beta_n)_{n\geq0}$, with
  $\beta_n>2(n+1)$ for every $n$, such that the branches $\lambda_n$ and
  $\lambda_{n+1}$ meet exactly once, at $b=\beta_n$. Moreover
  $\lambda(b)=\lambda_0(b)$ for $0\leq b\leq\beta_0$, and
  \[
   \lambda(b)=\lambda_n(b)
   \quad\text{for}\quad
   \beta_{n-1}\leq b\leq\beta_n, \quad n\geq1.
  \]
 \item[\textup{(HL2)}]
  The function $b\mapsto\lambda_0(b)/b$ is increasing on $(0,+\infty)$, and for
  every $n\geq1$ the function $b\mapsto\lambda_n(b)/b$ decreases to a unique
  minimum and then increases.
\end{enumerate}
\end{theorem}

Thus, for any $b\ge0$ not equal to one of the values $\beta_n$, we can define the
\emph{active angular momentum} $n(b)$ by the relation
$\lambda(b)=\lambda_{n(b)}(b)$. We call the intervals $(0,\beta_0)$ and
$(\beta_{n-1},\beta_{n})$, for $n\ge1$, \emph{active angular-momentum intervals}.

At the crossing of the branches $n$ and $n+1$, one has the Saint-James
identity~\cite{SaintJames} (a complete proof is given
in~\cite[Theorem~1.1]{MR4947381})
\begin{equation}\label{eq:saint-james}
 \bigl(\beta_n-(2n+1)\bigr)^2=4\lambda(\beta_n)+1 .
\end{equation}
This identity is elementary, and we rederive it in
Lemma~\ref{lem:intertwining} below.

\begin{figure}[tb]
  \centering
  \includegraphics[width=.75\textwidth,page=2]{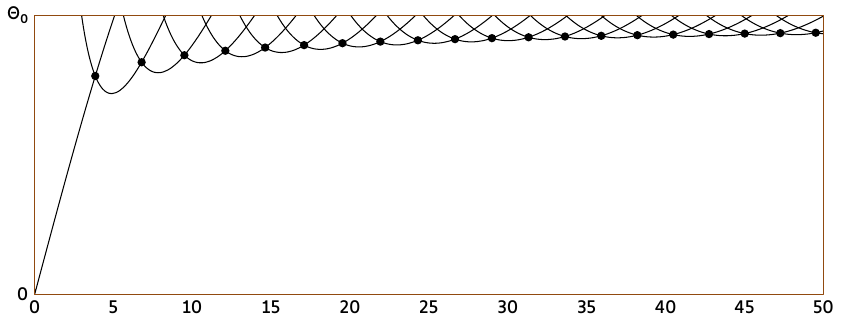}
  \caption{The normalized ground-state energy \(b\mapsto\lambda(b)/b\). The dots
  are the normalized crossing energies \(\eta_n^*\) of~\eqref{eq:eta-star-def};
  by Theorem~\ref{thm:crossing-quotients} they increase,
  and~\eqref{eq:de-gennes-limit} they converge to \(\Theta_0\), the horizontal line.
  Theorem~\ref{thm:de-gennes-bound} states that the curve stays strictly below
  that line.}
  \label{fig:quotient}
\end{figure}

\subsection{Main results}

Our first result extends strong diamagnetism in the disk, previously
known for large fields~\cite{MR2394546}, to every positive field
strength.

\begin{theorem}\label{thm:main}
  For every $0<b_1<b_2$, one has $\lambda(b_1)<\lambda(b_2)$.
\end{theorem}

To the best of our knowledge, the disk is the first bounded planar domain for
which strong diamagnetism has been established at every field strength.

At the crossings $\beta_n$ of Theorem~\ref{thm:HL} we define the
normalized crossing energies
\begin{equation}\label{eq:eta-star-def}
 \eta_n^*=\frac{\lambda(\beta_n)}{\beta_n}, \qquad n\geq0.
\end{equation}

\begin{theorem}\label{thm:crossing-quotients}
  The sequence $(\eta_n^*)_{n\geq0}$ is strictly increasing.
\end{theorem}

The first terms of the sequence $(\eta_n^*)_{n\geq0}$ are marked in
Figure~\ref{fig:quotient}.

The third result is the global comparison with the half-plane energy.

\begin{theorem}\label{thm:de-gennes-bound}
  For every $b>0$, one has
  \[
    \lambda(b)<\Theta_0b.
  \]
\end{theorem}

Theorem~\ref{thm:de-gennes-bound} has two conceptually different proofs, and we
give both. The first, in Section~\ref{sec:crossing-monotonicity}, is a
consequence of Theorem~\ref{thm:crossing-quotients}: by (HL2), on each active
angular-momentum interval the quotient $\lambda_n(b)/b$ has no interior maximum,
while its endpoint values are consecutive terms of $(\eta_n^*)_{n\geq0}$; since
$\beta_n\to\infty$, relation~\eqref{eq:de-gennes-limit} shows that this
increasing sequence converges to $\Theta_0$. The second, in
Section~\ref{sec:trial-state-proof}, is variational. The two proofs are not only
different in method but also in logical status: the first uses the strong-field
asymptotics~\eqref{eq:de-gennes-limit}, which is a nontrivial external input,
whereas the second does not.

The variational proof yields, in addition, the following explicit upper bound.
We denote by $\xi_0>0$ the minimizer of $\xi\mapsto\mu(\xi)$
in~\eqref{eq:Theta0}, by $\varphi_0$ the corresponding positive normalized
eigenfunction, and we set $C_1\coloneqq\varphi_0(0)^2/3$. We recall that
$\Theta_0=\xi_0^2$ (see for instance~\cite[Equation~(3.25)]{MR2662319}).

\begin{theorem}\label{thm:quantitative}
Set
\[
 B_1
 = \frac54+\frac53C_1\xi_0+\frac{37}{12}\xi_0^4,
 \qquad
 B_2
 = \frac53\xi_0^2\bigl(C_1+5\xi_0^3\bigr).
\]
Then, for every $b>4\Theta_0$,
\[
 \lambda(b)
 <
 \Theta_0b-C_1\sqrt b
 +\frac{B_1+B_2b^{-1/2}}{1-2\xi_0b^{-1/2}},
\]
and, for every $b\geq130$,
\[
 \lambda(b)<\Theta_0b-C_1\sqrt b+B_1+B_2b^{-1/2}.
\]
\end{theorem}

\begin{remark}
The bound is asymptotically sharp to order $\sqrt b$: the strong-field
asymptotics for smooth planar domains give
$\lambda(b)=\Theta_0b-C_1\sqrt b+o(\sqrt b)$ as $b\to+\infty$;
see~\cite[Theorem~8.3.2]{MR2662319}.
\end{remark}

Numerically, $C_1\approx\num{0.2540}$, $B_1\approx\num{2.649}$ and
$B_2\approx\num{2.479}$. To the best of our knowledge,
Theorem~\ref{thm:quantitative} is the first non-asymptotic version of
the two-term upper bound for $\lambda$ in the disk.

Theorem~\ref{thm:main} has a direct consequence for the third critical field in
the disk. In the usual nondimensionalization of Ginzburg--Landau theory, the
linearized normal-state threshold for the external magnetic field $H$ depends on
a parameter $\kappa>0$ and is determined by $\lambda(\kappa H)=\kappa^2$; see for
instance~\cite[Chapter~13]{MR2662319}. Following Fournais and
Helffer~\cite[Section~3]{MR2394546}, one associates with the Ginzburg--Landau
functional two pairs of critical fields: the \emph{global} fields $\underline
H_{C_3}(\kappa)\leq\overline H_{C_3}(\kappa)$, defined through the minimizers of
the functional, and the \emph{local}, or spectral, fields
\begin{equation}\label{eq:local-fields}
  \begin{gathered}
 \underline H_{C_3}^{\mathrm{loc}}(\kappa)
 =\inf\bigl\{H>0\colon\lambda(\kappa H)\geq\kappa^2\bigr\},
 \\
 \overline H_{C_3}^{\mathrm{loc}}(\kappa)
 =\inf\bigl\{H>0\colon\lambda(\kappa H')\geq\kappa^2
   \ \text{for all}\ H'>H\bigr\},
  \end{gathered}
\end{equation}
defined purely in terms of the linear problem. As observed
in~\cite[Section~3]{MR2394546}, the gap within the second pair is caused
precisely by the possible failure of $b\mapsto\lambda(b)$ to be invertible, that
is, by the possible lack of strict monotonicity. Theorem~\ref{thm:main} removes
that obstruction in the disk, at every field strength.

\begin{corollary}\label{cor:critical-field}
For every $\kappa>0$, the equation $\lambda(\kappa H)=\kappa^2$ has exactly one
solution $H>0$. Consequently,
$\underline H_{C_3}^{\mathrm{loc}}(\kappa)=\overline H_{C_3}^{\mathrm{loc}}(\kappa)$
for every $\kappa>0$.
\end{corollary}

\begin{proof}
The map $b\mapsto\lambda(b)$ is continuous by standard perturbation theory for
quadratic forms, $\lambda(0)=0$, and $\lambda(b)\to+\infty$. Hence, by
Theorem~\ref{thm:main}, $H\mapsto\lambda(\kappa H)$ is a continuous and strictly
increasing bijection from $(0,+\infty)$ onto itself, and $\lambda(\kappa
H)=\kappa^2$ has a unique solution $H_\kappa$. Strict monotonicity then gives
$\{H>0:\lambda(\kappa H)\geq\kappa^2\}=[H_\kappa,+\infty)$, and likewise
$\{H>0:\lambda(\kappa H')\geq\kappa^2\ \text{for all}\
H'>H\}=[H_\kappa,+\infty)$; both infima in~\eqref{eq:local-fields} are therefore
equal to $H_\kappa$.
\end{proof}

Corollary~\ref{cor:critical-field} is the analogue, for the disk and for every
value of $\kappa$, of~\cite[Proposition~3.2]{MR2394546}, where the same conclusion
is obtained, but only for $\kappa$ larger than some unspecified $\kappa_0 > 0$. We
stress that it concerns the local fields only. Their identification with the
global ones~\cite[Theorem~3.1]{MR2394546} rests on the nonlinear analysis
of~\cite{MR2231969} and is likewise known only for large $\kappa$; that
restriction is of a different nature, and is not lifted by
Theorem~\ref{thm:main}. In the disk, the local critical field
$H_{C_3}^{\mathrm{loc}}(\kappa)$ is thus well defined for every $\kappa>0$,
whereas its identification with $\underline H_{C_3}(\kappa)=\overline
H_{C_3}(\kappa)$ remains restricted to the large-$\kappa$ regime.

The limit $\lambda(b)\to+\infty$ as $b\to+\infty$ used here follows
from~\eqref{eq:de-gennes-limit}, but also, without any asymptotic input, from the
results proved below: by Lemma~\ref{lem:crossing}, with
$d_n=\beta_{n-1}/2-n$, we have $3d_n^2+d_n>2n$, so that $d_n\to+\infty$ and hence
$\lambda(\beta_{n-1})=d_n(d_n+1)\to+\infty$ by~\eqref{eq:crossing-energy}, while
$\beta_n\to+\infty$ by (HL1). Corollary~\ref{cor:critical-field} thus relies only
on Theorem~\ref{thm:HL} and on Theorem~\ref{thm:main}.

Finally, Theorem~\ref{thm:crossing-quotients} localizes the crossing fields
themselves.

\begin{corollary}\label{cor:crossing-fields}
For every $n\geq0$,
\[
 \beta_n
 <
 (2n+1)+2\Theta_0+\sqrt{4\Theta_0(2n+1)+4\Theta_0^2+1},
\]
and, for $n\geq1$,
\[
 \beta_n
 >
 (2n+1)+2\eta_0^*+\sqrt{4\eta_0^*(2n+1)+4(\eta_0^*)^2+1}.
\]
\end{corollary}

\begin{proof}
By~\eqref{eq:saint-james} and~\eqref{eq:eta-star-def}, $\beta_n$ solves
$(\beta-(2n+1))^2=4\eta_n^*\beta+1$, that is
\[
 \beta^2-\bigl(2(2n+1)+4\eta_n^*\bigr)\beta+(2n+1)^2-1=0.
\]
The two roots are
$(2n+1)+2\eta\pm\sqrt{4\eta(2n+1)+4\eta^2+1}$ with $\eta=\eta_n^*$, and the
smaller one is at most $(2n+1)+2\eta-1\leq2n+2\eta<2(n+1)<\beta_n$; hence
$\beta_n$ is the larger root. That root is strictly increasing in $\eta$, and
$\eta_0^*\leq\eta_n^*<\Theta_0$ by Theorems~\ref{thm:crossing-quotients}
and~\ref{thm:de-gennes-bound}, with strict inequality on the left for $n\geq1$.
\end{proof}

Both bounds are of the form $2n+2\sqrt{2\eta n}+O(1)$, with $\eta=\Theta_0$ and
$\eta=\eta_0^*$ respectively; they therefore locate $\beta_n$ with a relative
error $O(n^{-1/2})$.

\begin{remark}\label{rem:beta-asymptotics}
Since $\eta_n^*\to\Theta_0$ by~\eqref{eq:de-gennes-limit}, the exact relation
used in the proof above gives
\[
 \beta_n=2n+2\sqrt{2\Theta_0\,n}+o(\sqrt n),
 \qquad n\to+\infty,
\]
with $2\sqrt{2\Theta_0}\approx\num{2.173}$. This is the heuristic behind the
choice of the starting field $b_{\mathrm{ini}}$ in
Section~\ref{subsec:intermediate}.
\end{remark}

\subsection{Strategy and organization}

There are essentially two ideas behind the proofs of Theorems~\ref{thm:main}
and~\ref{thm:crossing-quotients}.

The first is an exact derivative identity. By the Feynman--Hellmann formula we
can express $\lambda_n'(b)$ in terms of the single number $f_{n,b}(1)^2$;
see~\eqref{eq:derivative-identity}. Monotonicity of the branch is thereby reduced
to an upper bound on the boundary value of its normalized eigenfunction.

The second is the confluent hypergeometric representation of the radial
solutions. In the Euler integral representation of Kummer's function, a Neumann
condition becomes a prescribed \emph{first moment} of an explicit positive
weight on $(0,1)$; two Neumann conditions, holding simultaneously at a crossing,
become two orthogonality relations. Positivity of the corresponding Gram
determinants then produces the inequalities we need. A pleasant feature of this
representation is that its parameter $\nu$ depends only on the normalized energy
$\lambda/b$, and not on the angular momentum; this is what allows us to compare
two different sectors, and two different field strengths, within one family of
special functions.

Section~\ref{sec:preliminaries} collects this representation and some of its corollaries.
Section~\ref{sec:field-monotonicity} proves Theorem~\ref{thm:main}.
Section~\ref{sec:crossing-monotonicity} proves
Theorem~\ref{thm:crossing-quotients} and derives
Theorem~\ref{thm:de-gennes-bound} from it.
Section~\ref{sec:trial-state-proof} gives the second, variational proof of
Theorem~\ref{thm:de-gennes-bound}, and proves
Theorem~\ref{thm:quantitative}. It is independent of
Sections~\ref{sec:field-monotonicity} and~\ref{sec:crossing-monotonicity}, and
of Theorem~\ref{thm:HL}.

The variational proof of Section~\ref{sec:trial-state-proof} was first obtained
in the preprint~\cite{LenaSundqvistBound}. It is reproduced here in full,
including the certification data and code in Appendix~\ref{sec:details}.

\subsection{Further questions}

The disk is the simplest domain for which the questions above can be asked, and
the arguments of
Sections~\ref{sec:preliminaries}--\ref{sec:crossing-monotonicity} rely on its
rotational symmetry through the angular momentum decomposition~\eqref{eq:decomposition}.
Three natural questions remain open. Does global strong
diamagnetism hold for every smooth, bounded, simply connected planar domain, or
at least for every convex one? Does the bound $\lambda(b)<\Theta_0b$ hold for
such domains at every field strength? For large $b$ it follows from the two-term
asymptotics of~\cite{MR2662319}, since the maximal curvature is positive; the
variational scheme of Section~\ref{sec:trial-state-proof} adapts to the
large-field regime, but the treatment of a bounded range of fields uses the
explicit radial trial functions. Finally, one may ask whether the monotonicity of
the normalized crossing energies of Theorem~\ref{thm:crossing-quotients} has an
analogue for the ball in $\R^3$.

\section{The confluent hypergeometric representation}\label{sec:preliminaries}

This section collects the elementary facts about the radial solutions that are
used in Sections~\ref{sec:field-monotonicity}
and~\ref{sec:crossing-monotonicity}. Throughout, $n\in\N_0$ and $b>0$, and we
write
\begin{equation}\label{eq:x-nu}
 x=\frac b2,
 \qquad
 \nu=\frac12\left(1-\frac\lambda b\right)
\end{equation}
for the rescaled field strength and the spectral parameter attached to an energy
$\lambda$. We emphasize that $\nu$ depends only on the normalized energy
$\lambda/b$, and not on the sector index $n$.

\subsection{A priori bounds}

\begin{lemma}\label{lem:apriori}
For every $b>0$ and $n\in\N_0$ one has $\lambda_n(b)>0$. If moreover $b>2n$,
then
\begin{equation}\label{eq:lambda-less-b}
 \lambda_n(b)<b .
\end{equation}
Consequently, if $b>2n$ and $\lambda=\lambda_n(b)$, then $0<\nu<1/2$
in~\eqref{eq:x-nu}.
\end{lemma}

\begin{proof}
The operator $H_{n,b}$ is strictly positive for $b>0$: vanishing of its
quadratic form would force both $f'=0$ and $(n/r-br/2)f=0$ almost everywhere,
and hence $f=0$. For the upper bound, let $u(r)=r^ne^{-br^2/4}$. A direct
computation shows that $u$ solves $H_{n,b}u=bu$ on $(0,1)$ and that
$u'(1)=(n-b/2)u(1)$. Integrating by parts,
\[
 q_{n,b}[u]-b\norm{u}^2
 =u(1)u'(1)
 =\left(n-\frac b2\right)e^{-b/2}<0
 \qquad (b>2n),
\]
where the norm is taken in $L^2((0,1),r\dd r)$. The variational
characterization of $\lambda_n(b)$ gives~\eqref{eq:lambda-less-b}. The last
assertion is immediate from~\eqref{eq:x-nu}.
\end{proof}

In the notation of Lemma~\ref{lem:regular-solution} below, the trial state used
here is $u=u_{n,x,0}$: it is the member $\nu=0$, corresponding to $\lambda=b$, of the
family~\eqref{eq:kummer-solution}, since $M(0,n+1,\cdot)\equiv1$. The same
function reappears in~\eqref{eq:R-at-one}.

\subsection{The regular solution and the Neumann condition}

\begin{lemma}\label{lem:regular-solution}
Let $b>0$, $n\in\N_0$ and $\lambda\in\R$, and let $x$ and $\nu$ be as
in~\eqref{eq:x-nu}. Up to a multiplicative constant, the unique solution of
$H_{n,b}f=\lambda f$ on $(0,1)$ belonging to the form domain near the origin is
\begin{equation}\label{eq:kummer-solution}
 u_{n,x,\nu}(r)=r^n e^{-xr^2/2}M(\nu,n+1,xr^2),
\end{equation}
where $M$ denotes Kummer's function. Writing $F=M(\nu,n+1,\cdot)$, the Neumann
condition $u_{n,x,\nu}'(1)=0$ is equivalent to
\begin{equation}\label{eq:kummer-log-derivative}
 \frac{F'(x)}{F(x)}=\frac {x-n}{2x} .
\end{equation}
\end{lemma}

\begin{proof}
Substituting $z=xr^2$ and $f(r)=r^ne^{-z/2}w(z)$ in
\[
 -f''-\frac1rf'+\left(\frac nr-xr\right)^2f=\lambda f,
\]
dividing by $4x$, and recalling~\eqref{eq:x-nu}, turns the equation into Kummer's
equation
\[
 zw''+(n+1-z)w'-\nu w=0.
\]
The solution regular at the origin is $w=M(\nu,n+1,\cdot)$, the second solution
being singular there (see~\cite[Section~13.2]{DLMF}). This
gives~\eqref{eq:kummer-solution}. Differentiating~\eqref{eq:kummer-solution} at
$r=1$ gives
\begin{equation}\label{eq:log-derivative-at-one}
 \frac{u_{n,x,\nu}'(1)}{u_{n,x,\nu}(1)}
 =n-x+2x\frac{F'(x)}{F(x)},
\end{equation}
which vanishes precisely when~\eqref{eq:kummer-log-derivative} holds.
\end{proof}

\subsection{The Euler representation}

Assume $0<\nu<n+1$ and set $F=M(\nu,n+1,\cdot)$. Then, with a constant
$c_{\nu,n}>0$,
\begin{equation}\label{eq:kummer-integral}
 F(z)
 =c_{\nu,n}\int_0^1e^{zt}t^{\nu-1}(1-t)^{n-\nu}\dd t ;
\end{equation}
see~\cite[Equation~13.4.1]{DLMF}.

\begin{lemma}\label{lem:euler}
With $0<\nu<n+1$ and $F$ as above, let
\[
 w_{n,x,\nu}(t)=
 \frac{e^{xt}t^{\nu-1}(1-t)^{n-\nu}}
 {\int_0^1e^{xs}s^{\nu-1}(1-s)^{n-\nu}\dd s},
 \qquad 0<t<1,
\]
be the associated probability density on $(0,1)$. Then
\begin{equation}\label{eq:euler-consequences}
 \frac{F(z)}{F(x)}=\int_0^1e^{-(x-z)t}w_{n,x,\nu}(t)\dd t,
 \qquad
 \frac{F'(x)}{F(x)}=\int_0^1 t w_{n,x,\nu}(t)\dd t .
\end{equation}
In particular $F>0$ and $F'>0$ on $(0,+\infty)$, and by
Lemma~\ref{lem:regular-solution} the Neumann condition $u_{n,x,\nu}'(1)=0$ holds
if and only if the first moment of $w_{n,x,\nu}$ equals $(x-n)/(2x)$. Finally,
\begin{equation}\label{eq:weight-shift}
 w_{n+1,x,\nu}(t)\ \text{ is proportional to }\ (1-t) w_{n,x,\nu}(t),
\end{equation}
so that the weights for two consecutive values of $n$, at the same field strength and
the same normalized energy, differ by the factor $1-t$.
\end{lemma}

\begin{proof}
Everything is a direct consequence of~\eqref{eq:kummer-integral}: the two
identities in~\eqref{eq:euler-consequences} follow by inserting the
representation and dividing, and~\eqref{eq:weight-shift} follows by comparing
the exponents of $1-t$ for the parameter pairs $(\nu,n+1)$ and $(\nu,n+2)$.
\end{proof}

\subsection{The intertwining operator and the Saint-James identity}

\begin{lemma}\label{lem:intertwining}
Let $x>0$, $n\in\N_0$ and $0\leq\nu<n+1$, and set
\[
 \mathcal C_{n,x}=\frac d{dr}-\frac nr-xr .
\]
Then
\begin{equation}\label{eq:intertwining}
 \mathcal C_{n,x}u_{n,x,\nu}
 =-\frac{2x(n+1-\nu)}{n+1} u_{n+1,x,\nu} .
\end{equation}
Moreover, let $b=2x$ and let $\lambda=b(1-2\nu)$ be the energy attached to $\nu$
through~\eqref{eq:x-nu}. If $u=u_{n,x,\nu}$ satisfies $u'(1)=0$, then
$v=\mathcal C_{n,x}u$ satisfies
\begin{equation}\label{eq:v-prime}
 v'(1)=\bigl((x-n)(x-n-1)-\lambda\bigr)u(1) .
\end{equation}
Consequently, for $0<\nu<1/2$ the following are equivalent:
\begin{enumerate}[(i)]
 \item $u_{n,x,\nu}'(1)=0$ and $(x-n)(x-n-1)=\lambda$;
 \item the branches $\lambda_n$ and $\lambda_{n+1}$ cross at
  $b=2x$, with common value $\lambda$.
\end{enumerate}
In particular, at a crossing one has $(x-n)(x-n-1)=\lambda$, which is
exactly~\eqref{eq:saint-james}.
\end{lemma}

\begin{proof}
Writing $F=M(\nu,n+1,\cdot)$, a direct computation gives
\[
 \mathcal C_{n,x}u_{n,x,\nu}(r)
 =2xr^{n+1}e^{-xr^2/2}\bigl[F'(xr^2)-F(xr^2)\bigr],
\]
while~\eqref{eq:kummer-integral} yields, for $0<\nu<n+1$,
\[
 F'(z)-F(z)
 =-c_{\nu,n}\int_0^1e^{zt}t^{\nu-1}(1-t)^{n+1-\nu}\dd t
 =-\frac{n+1-\nu}{n+1}M(\nu,n+2,z);
\]
this is~\cite[Equation~13.3.10]{DLMF}, and it extends to $\nu=0$ by continuity.
This proves~\eqref{eq:intertwining}. Next, the radial equation gives
$u''(1)=-u'(1)+\bigl((n-x)^2-\lambda\bigr)u(1)$, so that, when $u'(1)=0$,
\[
 v'(1)
 =u''(1)-(n+x)u'(1)+(n-x)u(1)
 =\bigl((n-x)^2+(n-x)-\lambda\bigr)u(1),
\]
which is~\eqref{eq:v-prime}.

Assume now $0<\nu<1/2$, so that $M(\nu,n+1,\cdot)$ and $M(\nu,n+2,\cdot)$ are
positive by~\eqref{eq:kummer-integral}.

Suppose (i). Then $u=u_{n,x,\nu}$ is a positive Neumann eigenfunction of
$H_{n,b}$ with eigenvalue $\lambda$, hence the sector-$n$ ground state, and
$\lambda_n(b)=\lambda$. By~\eqref{eq:intertwining}, $v=\mathcal C_{n,x}u$ is a
nonzero regular solution of one sign in sector $n+1$ with the same energy, and
by~\eqref{eq:v-prime} and $(x-n)(x-n-1)=\lambda$ it satisfies $v'(1)=0$. Hence
$v$ is the sector-$(n+1)$ ground state and $\lambda_{n+1}(b)=\lambda$, so the
two branches cross at $b$.

Conversely, suppose (ii). The sector-$n$ ground state is a positive multiple of
$u=u_{n,x,\nu}$ and satisfies $u'(1)=0$; by~\eqref{eq:intertwining} the
sector-$(n+1)$ ground state is a nonzero multiple of $v=\mathcal C_{n,x}u$, and
its Neumann condition together with~\eqref{eq:v-prime} gives
$(x-n)(x-n-1)=\lambda$, which is just~\eqref{eq:saint-james} in disguise.
\end{proof}

\begin{remark}\label{rem:saint-james-consequence}
Let $n\geq1$. By (HL1) we have $\beta_{n-1}>2n$. We may record, for later
use, that at the lower crossing of the $n$-th branch
\begin{equation}\label{eq:crossing-energy}
 \lambda_n(\beta_{n-1})=d_n(d_n+1),
 \qquad\text{where }
 d_n\coloneqq\frac{\beta_{n-1}}2-n>0 .
\end{equation}
Indeed, by Lemma~\ref{lem:intertwining} applied with $n$ replaced by $n-1$ and
$x=\beta_{n-1}/2=n+d_n$, the crossing energy equals
$(x-n+1)(x-n)=d_n(d_n+1)$.
\end{remark}

\section{Strict monotonicity in the field strength}\label{sec:field-monotonicity}

By Theorem~\ref{thm:HL}, the branch $\lambda_n$ realizes the ground-state
energy exactly on $[\beta_{n-1},\beta_n]$, and $\lambda$ is continuous. It is
therefore enough to prove that
\begin{equation}\label{eq:goal-monotonicity}
 \lambda_0'(b)>0\ \ (b>0),
 \qquad
 \lambda_n'(b)>0\ \ (b\geq\beta_{n-1},\ n\geq1).
\end{equation}
For $n=0$ this is immediate: since
$\partial_b(n/r-br/2)^2=br^2/2$ when $n=0$, the Feynman--Hellmann formula
gives
\begin{equation}
 \lambda_0'(b)
 =\frac b2\int_0^1r^2f_{0,b}(r)^2r\dd r>0
 \qquad (b>0).
\end{equation}
We therefore fix $n\geq1$ for the rest of this section.

The proof proceeds in four steps. An exact derivative identity
(Section~\ref{subsec:derivative-identity}) reduces positivity of $\lambda_n'$ to
an upper bound on the boundary value $f_{n,b}(1)^2$. The Euler representation of
Section~\ref{sec:preliminaries} provides such a bound, uniform in $b$, in terms
of an explicit integral $\J_n$ (Section~\ref{subsec:trace}). Two orthogonality
relations available at the lower crossing $\beta_{n-1}$ give two-sided control
of the crossing parameter $d_n$ (Section~\ref{subsec:moments}), and a Hankel
determinant converts that control into the matching lower bound on $\J_n$
(Section~\ref{subsec:Jn}). These estimates fit together in
Section~\ref{subsec:positivity}.

\subsection{An exact derivative identity}\label{subsec:derivative-identity}

\begin{lemma}\label{lem:derivative-identity}
For every $n\geq0$ and $b>0$,
\begin{equation}\label{eq:derivative-identity}
 \lambda_n'(b)
 =\frac{\lambda_n(b)}{b}
  -\frac{\lambda_n(b)-(b/2-n)^2}{2b} f_{n,b}(1)^2.
\end{equation}
\end{lemma}

\begin{proof}
Write $f=f_{n,b}$ and $\lambda=\lambda_n(b)$. Since
$\partial_b(n/r-br/2)^2=br^2/2-n$, the Feynman--Hellmann formula gives
\begin{equation}\label{eq:feynman-hellmann}
 \lambda_n'(b)
 =\int_0^1\left(\frac b2r^2-n\right)f(r)^2r\dd r.
\end{equation}
Write the eigenvalue equation as
\[
 -(rf')'+r\left(\frac nr-\frac b2r\right)^2f=\lambda rf,
\]
multiply it by $rf'$, and integrate over $(0,1)$. The first term contributes
\[
 -\int_0^1(rf')'(rf')\dd r=-\frac12\Bigl[(rf')^2\Bigr]_0^1=0,
\]
since $f'(1)=0$ and $rf'(r)\to0$ as $r\to0$ (because $f(r)=O(r^n)$ near the
origin if $n\geq1$, and because $f'(0)=0$ if $n=0$). For the second term,
integrating by parts and using $[(n-br^2/2)^2]'=-2br(n-br^2/2)$,
\[
 \begin{aligned}
 \int_0^1r^2\left(\frac nr-\frac b2r\right)^2ff'\dd r
 &
 =\frac12\Bigl[\left(n-\frac b2r^2\right)^2f^2\Bigr]_0^1
  +b\int_0^1\left(n-\frac b2r^2\right)f^2r\dd r
 \\
 &
 =\frac12\left(\frac b2-n\right)^2f(1)^2-b\lambda_n'(b),
 \end{aligned}
\]
the last equality by~\eqref{eq:feynman-hellmann}. For the third term, the
normalization of $f$ gives
\[
 \lambda\int_0^1r^2ff'\dd r
 =\frac\lambda2\Bigl[r^2f^2\Bigr]_0^1-\lambda\int_0^1f^2r\dd r
 =\frac\lambda2f(1)^2-\lambda .
\]
(The boundary terms at the origin vanish in both cases: for $n\geq1$ because
$f(r)=O(r^n)$, and for $n=0$ because $r^2f^2$ and $(n-br^2/2)^2$ vanish there.)
The three displays together imply~\eqref{eq:derivative-identity}.
\end{proof}

It follows in particular that if $\lambda_n(b)\leq (b/2-n)^2$, then
\begin{equation}\label{eq:easy-case}
 \lambda_n'(b)\geq\frac{\lambda_n(b)}{b}>0 .
\end{equation}
The whole difficulty therefore lies in the regime
$\lambda_n(b)>(b/2-n)^2$, where the two terms in~\eqref{eq:derivative-identity}
compete.

\subsection{A uniform boundary trace estimate}\label{subsec:trace}

For $n\geq1$, define
\[
 \J_n=\int_0^1 t^n e^{n(1-t)}\dd t.
\]

\begin{lemma}\label{lem:trace}
If $b>2n$, then
\[
 f_{n,b}(1)^2<\frac2{\J_n}.
\]
\end{lemma}

\begin{proof}
Let $x=b/2$ and let $\nu$ be attached to $\lambda=\lambda_n(b)$
by~\eqref{eq:x-nu}; by Lemma~\ref{lem:apriori}, $0<\nu<1/2$, so all the Euler
integrals of Lemma~\ref{lem:euler} converge. By
Lemma~\ref{lem:regular-solution}, $f_{n,b}$ is a positive multiple of
$u_{n,x,\nu}$; we write $f=f_{n,b}$, $F=M(\nu,n+1,\cdot)$ and
$w=w_{n,x,\nu}$.

Let $0<r<1$ and put $y=x(1-r^2)$, so that $x-y=xr^2$. Since $t\mapsto e^{-yt}$
is strictly convex and $w$ is positive on $(0,1)$,
Jensen's inequality and~\eqref{eq:euler-consequences} give
\[
 \frac{F(x-y)}{F(x)}
 =\int_0^1e^{-yt}w(t)\dd t
 >\exp\left(-y\int_0^1tw(t)\dd t\right)
 =\exp\left(-\frac{y(x-n)}{2x}\right),
\]
the last equality being the Neumann
condition~\eqref{eq:kummer-log-derivative}. Hence, by~\eqref{eq:kummer-solution},
\[
 \frac{f(r)}{f(1)}
 =r^n e^{y/2}\frac{F(x-y)}{F(x)}
 >r^n\exp\left(\frac n2(1-r^2)\right).
\]
The normalization of $f$ now gives
\[
 1
 =f(1)^2\int_0^1
   \left(\frac{f(r)}{f(1)}\right)^2r\dd r
 >f(1)^2\int_0^1r^{2n+1}e^{n(1-r^2)}\dd r
 =\frac12f(1)^2\J_n ,
\]
by the substitution $t=r^2$. This proves the claim.
\end{proof}

\subsection{Moment inequalities at the lower crossing}\label{subsec:moments}

Recall from Remark~\ref{rem:saint-james-consequence} that
$d_n=\beta_{n-1}/2-n>0$ and that
$\lambda_n(\beta_{n-1})=d_n(d_n+1)$.

\begin{lemma}\label{lem:crossing}
For every $n\geq1$,
\[
 d_n(d_n+1)<2n<3d_n^2+d_n.
\]
\end{lemma}

\begin{proof}
Write $d=d_n$ and
\[
 x=\frac {\beta_{n-1}}{2}=n+d,
 \quad \nu=\frac{1}{2}-\frac{d(d+1)}{4x},
\]
so that $\nu$ is attached to $\lambda=\lambda_n(\beta_{n-1})=d(d+1)$
by~\eqref{eq:x-nu}. Since $\beta_{n-1}>2n$, Lemma~\ref{lem:apriori} gives
$0<\nu<1/2$, so all the Euler integrals used below are convergent. Recall
from~\eqref{eq:x-nu} that $\nu$ does not depend on the sector index: at the
crossing, the positive radial eigenfunctions in sectors $n-1$ and $n$ are
proportional, respectively, to $u_{n-1,x,\nu}$ and $u_{n,x,\nu}$.

Define
\[
 w(t)=e^{xt}t^{\nu-1}(1-t)^{n-\nu-1},
 \qquad 0<t<1,
\]
which is proportional to $w_{n-1,x,\nu}$, and, for $k\geq0$, set
\[
 I_k=\int_0^1t^k w(t)\dd t,
 \qquad
 m_k=\frac{I_k}{I_0}.
\]
By~\eqref{eq:weight-shift}, the Euler weight for the sector-$n$ solution is
proportional to $(1-t)w(t)$, and its normalized first moment is therefore
$(m_1-m_2)/(1-m_1)$. By Lemma~\ref{lem:euler}, the Neumann conditions in sectors
$n-1$ and $n$ read, respectively,
\[
 \frac{I_1}{I_0}=\frac{x-(n-1)}{2x}=\frac{d+1}{2x},
 \qquad
 \frac{m_1-m_2}{1-m_1}=\frac{x-n}{2x}=\frac {d}{2x}.
\]
Consequently,
\begin{equation}\label{eq:m1-m2}
 m_1=\frac{d+1}{2x},
 \qquad
 m_2=\frac{d^2+3d+2n}{4x^2}.
\end{equation}

For $k\geq0$, integration of the derivative of $e^{xt}t^{\nu+k}(1-t)^{n-\nu}$
over $(0,1)$ gives (the boundary terms vanish because $0<\nu<1/2$ and $n\geq 1$)
\begin{equation}\label{eq:moment-recurrence}
 x m_{k+2}=(\nu+k)m_k+(d-k)m_{k+1}.
\end{equation}
Let
\[
 \langle p,q\rangle=\frac1{I_0}\int_0^1p(t)q(t)w(t)\dd t,
\]
and put
\[
 \Pi(t)=t^2-\frac{2d}{x}t+\frac{3d^2+d-2n}{4x^2}.
\]
Equations~\eqref{eq:m1-m2} and~\eqref{eq:moment-recurrence}, with $k=1$, give
\[
 \langle \Pi,1\rangle=\langle \Pi,t\rangle=0.
\]

We claim that the monic quadratic $\Pi$ has two simple zeros in $(0,1)$. Indeed,
from $\langle \Pi,1\rangle = 0$ we find that it must have at least one zero in
$(0,1)$. If there were only one simple zero in $(0,1)$, say $t=t_1$, it would
hold that $\langle \Pi,t-t_1\rangle\neq0$. If there were a double zero, then we
would still get $\langle \Pi,1\rangle\neq 0$. The product of the zeros is
\[
 \Pi(0)=\frac{3d^2+d-2n}{4x^2}>0,
\]
which proves $2n<3d^2+d$.

Since $\Pi$ is orthogonal to $1$ and $t$, a further use
of~\eqref{eq:moment-recurrence}, for $k=1,2$, gives
\[
 0<\langle \Pi,\Pi\rangle
 =\langle t^2,\Pi\rangle
 =\frac{(2n+d-1)(2n-d^2-d)}{4x^4}.
\]
The first factor is positive, and hence $d(d+1)<2n$.
\end{proof}

\subsection{Lower bounds on $\J_n$}\label{subsec:Jn}

We first bound $\J_n$ from below by an expression which involves $n$ alone,
and then compare that expression with $1/(d_n+1)$.

\begin{lemma}\label{lem:Jn-lower}
For every integer $n\geq1$ one has $Q_n(\J_n)>0$, where
\[
 Q_n(t)=2(n+1)^2t^2+2(n+1)t-(3n+4) ;
\]
equivalently,
\begin{equation}\label{eq:Jn-lower}
 \J_n>\frac{\sqrt{6n+9}-1}{2(n+1)} .
\end{equation}
\end{lemma}

\begin{proof}
Fix an integer $n\geq1$ and put
\[
 h(u)=(1-u)e^u,
 \qquad
 A_k=\int_0^1u^k h(u)^n\dd u.
\]
The substitution $t=1-u$ shows that $A_0=\J_n$. Since $h'(u)/h(u)=-u/(1-u)$,
integration of the derivative of $u^k(1-u)h(u)^n$ over $(0,1)$ gives
\begin{equation}\label{eq:A-recurrence}
 A_0+nA_1=1,
 \qquad
 kA_{k-1}-(k+1)A_k-nA_{k+1}=0 \quad (k\geq1).
\end{equation}
The matrix $(A_{i+j})_{0\leq i,j\leq2}$ is positive definite, since it is the
Gram matrix of $1,u,u^2$ in $L^2((0,1),h(u)^n\dd u)$. Successive use
of~\eqref{eq:A-recurrence} gives
\[
 \det(A_{i+j})_{0\leq i,j\leq2}
 =\frac{\bigl((n+2)A_0-2\bigr)
 \bigl(2(n+1)^2A_0^2+2(n+1)A_0-(3n+4)\bigr)}{n^6}.
\]
The first factor is $n^2A_2>0$, whence $Q_n(\J_n)>0$. Finally, the discriminant
of $Q_n$ is $4(n+1)^2+8(n+1)^2(3n+4)=4(n+1)^2(6n+9)$, so that the positive root
of $Q_n$ is the right-hand side of~\eqref{eq:Jn-lower}; since $\J_n>0$, the two
statements are equivalent.
\end{proof}

\begin{remark}\label{rem:two-by-two}
The same computation with the $2\times2$ Gram matrix $(A_{i+j})_{0\leq i,j\leq1}$
gives, by~\eqref{eq:A-recurrence},
\[
 A_0A_2-A_1^2=\frac{(n+1)A_0^2-1}{n^2}>0,
 \qquad\text{that is,}\qquad
 \J_n>\frac1{\sqrt{n+1}} .
\]
This weaker bound is not sufficient for Lemma~\ref{lem:Jn} below; see
Remark~\ref{rem:sharpness}.
\end{remark}

\begin{lemma}\label{lem:Jn}
For every integer $n\geq1$,
\begin{equation}\label{eq:key-inequality}
 (d_n+1)\J_n>1.
\end{equation}
\end{lemma}

\begin{proof}
Write $d=d_n$. Since $Q_n$ is strictly increasing on $(0,+\infty)$ and
$Q_n(\J_n)>0$ by Lemma~\ref{lem:Jn-lower}, it suffices to prove that
$Q_n\bigl(1/(d+1)\bigr)<0$. Set $N = (3d^2+d)/2$ and
\[
 G(s)=(d+1)^2 Q_s\bigl(1/(d+1)\bigr)
 =2(s+1)^2+2(s+1)(d+1)-(3s+4)(d+1)^2 ,
\]
where $Q_s$ is defined by the same formula as $Q_n$ with $n$ replaced by the
real parameter $s$. The second inequality in Lemma~\ref{lem:crossing} gives
$0<n<N$. The polynomial $G$ is convex, and
\[
 G(0)=-2d(2d+3)<0,
 \qquad
 G(N)=-\frac {d}{2}(9d^2+2d+9)<0 .
\]
Hence $G(n)<0$, that is $Q_n(1/(d+1))<0$, and the conclusion follows.
\end{proof}

\begin{remark}\label{rem:sharpness}
The two ingredients of this proof are asymptotically matched to leading order, so
that neither can be weakened. On the one hand, Laplace's method applied to
$\J_n=\int_0^1e^{n(\log t+1-t)}\dd t$ gives $\J_n=\sqrt{\pi/(2n)}(1+o(1))$,
while the right-hand side of~\eqref{eq:Jn-lower} is
$\bigl(\sqrt6/2\bigr)n^{-1/2}(1+o(1))$: the bound of Lemma~\ref{lem:Jn-lower} is
off by the factor $\sqrt{\pi/3}\approx\num{1.023}$ only. On the other hand the
second inequality of Lemma~\ref{lem:crossing} gives $d_n>\sqrt{2n/3}(1+o(1))$,
and
\[
 \sqrt{\frac23}\cdot\frac{\sqrt6}2=1 .
\]
Along this chain of estimates,~\eqref{eq:key-inequality} therefore holds only
by virtue of the lower-order terms. With the bound of
Remark~\ref{rem:two-by-two} in place of~\eqref{eq:Jn-lower}, one would obtain
$(d_n+1)\J_n>\sqrt{2/3}(1+o(1))$ instead, which is not enough. The
inequality~\eqref{eq:key-inequality} itself is not sharp: by
Remark~\ref{rem:beta-asymptotics} one has $d_n=\sqrt{2\Theta_0n}(1+o(1))$ and
therefore $(d_n+1)\J_n\to\sqrt{\pi\Theta_0}\approx\num{1.362}$, the loss being
entirely in the second inequality of Lemma~\ref{lem:crossing}.
\end{remark}

\subsection{Strict increase of every active branch}\label{subsec:positivity}

Fix $n\geq1$ and let $b\geq\beta_{n-1}$; write $d=b/2-n$, so that $d\geq d_n>0$.
Since
\[
 \frac{\partial}{\partial b}
 \left[\left(\frac {n}{r}-\frac{br}{2}\right)^2-(b/2-n)^2\right]
 =\frac{b}{2}(r^2-1)\leq0
 \qquad (0<r<1,\ b>0),
\]
the quadratic form $q_{n,b}-(b/2-n)^2\norm{\cdot}^2$, whose domain does not
depend on $b$, is nonincreasing in $b$ for each fixed function. The min--max
principle therefore shows that
\[
 b\mapsto \lambda_n(b)-(b/2-n)^2
\]
is nonincreasing on $(0,+\infty)$. At $b=\beta_{n-1}$,
equation~\eqref{eq:crossing-energy} gives
\[
 \lambda_n(\beta_{n-1})-\left(\frac{\beta_{n-1}}2-n\right)^2
 =d_n(d_n+1)-d_n^2=d_n .
\]
Hence, for $b\geq\beta_{n-1}$,
\[
  \lambda_n(b)-(b/2-n)^2 \leq d_n .
\]
If the left-hand side is nonpositive, then $\lambda_n'(b)>0$ follows at once
from~\eqref{eq:easy-case}. We therefore assume that it is positive, and set
$D=\lambda_n(b)-(b/2-n)^2$, so that
\[
0 < D \leq d_n \leq d .
\]
Since $b\geq\beta_{n-1}>2n$, Lemma~\ref{lem:trace} and Lemma~\ref{lem:Jn} give
\[
 f_{n,b}(1)^2<\frac{2}{\J_n}<2(d_n+1)\leq2(d+1).
\]
Multiplying~\eqref{eq:derivative-identity} by $2b$ and using
$\lambda_n(b)=D+d^2$, we obtain
\[
 2b\lambda_n'(b)
 =2\lambda_n(b)-D f_{n,b}(1)^2
 >2\lambda_n(b)-2D(d+1)
 = 2d(d-D)
 \geq 0.
\]
Thus $\lambda_n'(b)>0$ for every $b\geq\beta_{n-1}$, which
proves~\eqref{eq:goal-monotonicity}.

\begin{proof}[Proof of Theorem~\ref{thm:main}]
By~\eqref{eq:goal-monotonicity}, $\lambda_0$ is strictly increasing on
$[0,+\infty)$, and for every $n\geq1$ the branch $\lambda_n$ is strictly
increasing on the whole interval on which it can realize the ground energy,
namely from its lower crossing onward. By (HL1) the global ground energy is
obtained successively from these branches, and adjacent branches agree at each
crossing; since $\beta_n>2(n+1)$, only finitely many crossings occur in a
bounded $b$-interval. Combining the strict increase across those finitely many
intervals gives $\lambda(b_1)<\lambda(b_2)$ whenever $0<b_1<b_2$.
\end{proof}

\section{Monotonicity of the normalized crossing energies}\label{sec:crossing-monotonicity}

We now prove Theorem~\ref{thm:crossing-quotients}, and then deduce
Theorem~\ref{thm:de-gennes-bound} from it.

The argument runs as follows. We first show that the crossings of $\lambda_n$
and $\lambda_{n+1}$ are exactly the zeros in $(0,1)$ of an explicit function
$\mathcal R_n$ of the normalized energy $s$ alone
(Section~\ref{subsec:zeros}); by (HL1) there is exactly one such zero, namely
$\eta_n^*$, and $\mathcal R_n(1)<0$. It therefore suffices to prove
\[
 \mathcal R_n(\eta_{n-1}^*)>0 ,
\]
for the intermediate value theorem then places $\eta_n^*$ strictly to the right
of $\eta_{n-1}^*$. To do so we \emph{freeze} the normalized energy at the value
$\eta=\eta_{n-1}^*$ and let the field vary. Freezing $\eta$ freezes the
parameter $\nu$ of Section~\ref{sec:preliminaries}, so a single Kummer function
governs both the actual crossing $\beta_{n-1}$ and the field $\widehat x$ at
which the \emph{next} crossing would occur if the normalized energy did not
increase (Section~\ref{subsec:comparison}). The logarithmic derivative of that
Kummer function obeys a Riccati equation, and a convexity estimate shows that it
overshoots between the two fields (Section~\ref{subsec:convexity}). That
overshoot is precisely $\mathcal R_n(\eta)>0$.

Throughout this section we fix $n\geq1$ and write, in the notation of
Remark~\ref{rem:saint-james-consequence},
\begin{equation}\label{eq:section4-notation}
 d=d_n, \qquad x=n+d=\frac{\beta_{n-1}}2, \qquad
 \eta=\eta_{n-1}^*=\frac{\lambda(\beta_{n-1})}{\beta_{n-1}}=\frac{d(d+1)}{2x}.
\end{equation}
The first inequality in Lemma~\ref{lem:crossing} gives
\begin{equation}\label{eq:eta-x-less-n}
 \eta x=\frac{d(d+1)}2<n,
\end{equation}
and since $x>n$ it follows that $0<\eta<1$.

\subsection{Crossings as zeros of an explicit function}\label{subsec:zeros}

For $0<s\leq1$, let $z_n(s)$ be the larger root of
\begin{equation}\label{eq:saint-james-z}
 z^2-(2s+2n+1)z+n(n+1)=0 .
\end{equation}
The left-hand side of~\eqref{eq:saint-james-z}, evaluated at $z=n+1$, equals
$-2s(n+1)<0$; hence $z_n(s)>n+1$, and $z_n$ is real-analytic on $(0,1]$. Put
\[
 \nu_s=\frac{1-s}2\in\left[0,\frac12\right),
 \qquad
 u_{n,s}=u_{n,z_n(s),\nu_s},
 \qquad
 \mathcal R_n(s)=\frac{u_{n,s}'(1)}{u_{n,s}(1)} .
\]
By Lemma~\ref{lem:regular-solution}, $u_{n,s}$ is the regular solution of
$H_{n,2z_n(s)}u=2z_n(s)s u$; note that $\nu_s$ and $z_n(s)$ are exactly the
data attached by~\eqref{eq:x-nu} to the field $b=2z_n(s)$ and the energy
$\lambda=sb$. Since $M(\nu_s,n+1,\cdot)>0$, the function $\mathcal R_n$ is
well defined and real-analytic on $(0,1]$, and by~\eqref{eq:log-derivative-at-one}
\begin{equation}\label{eq:R-as-Y}
 \mathcal R_n(s)
 =2z_n(s)\frac{F_s'(z_n(s))}{F_s(z_n(s))}-\bigl(z_n(s)-n\bigr),
 \qquad F_s=M(\nu_s,n+1,\cdot).
\end{equation}

Since $M(0,n+1,\cdot)\equiv1$, we have $u_{n,1}(r)=r^ne^{-z_n(1)r^2/2}$ and
therefore
\begin{equation}\label{eq:R-at-one}
 \mathcal R_n(1)=n-z_n(1)<0 .
\end{equation}

\begin{lemma}\label{lem:zeros-of-R}
The zeros of $\mathcal R_n$ in $(0,1)$ are precisely the normalized energies of
the crossings of $\lambda_n$ and $\lambda_{n+1}$. Consequently $\mathcal R_n$
has exactly one zero in $(0,1)$, namely $\eta_n^*$.
\end{lemma}

\begin{proof}
Let $0<s<1$, so that $0<\nu_s<1/2$, and put $z=z_n(s)$ and
$\lambda=2zs$. By construction, $z$ and $\nu_s$ are the data attached
by~\eqref{eq:x-nu} to the field $b=2z$ and the energy $\lambda$, and
equation~\eqref{eq:saint-james-z} says exactly that $(z-n)(z-n-1)=\lambda$.
Hence condition (i) of Lemma~\ref{lem:intertwining} holds if and only if
$u_{n,s}'(1)=0$, that is, if and only if $\mathcal R_n(s)=0$; and condition (ii)
of that lemma says that $\lambda_n$ and $\lambda_{n+1}$ cross at $b=2z$ with
normalized energy $s$. Since a crossing field satisfies $b>2(n+1)$ by (HL1), the
associated $z$ exceeds $n+1$ and is therefore the larger root
of~\eqref{eq:saint-james-z}, so that every crossing with normalized energy $s$
arises this way. The two conditions being equivalent, the first assertion
follows.

By (HL1) there is exactly one crossing, at $b=\beta_n$; its normalized energy is
$\eta_n^*$, and $\eta_n^*\in(0,1)$ by Lemma~\ref{lem:apriori} applied to sector
$n+1$ at $b=\beta_n$. This gives the second assertion.
\end{proof}

It remains to prove that $\mathcal R_n(\eta)>0$.

\subsection{The comparison point}\label{subsec:comparison}

Set
\[
 \widehat{x}=z_n(\eta).
\]
Thus $\widehat x$ is the (rescaled) field at which the branches $\lambda_n$ and
$\lambda_{n+1}$ would cross if the normalized energy at that crossing were still
equal to $\eta$. Equation~\eqref{eq:saint-james-z} is equivalently
\begin{equation}\label{eq:next-saint-james}
 (\widehat{x}-n-1)(\widehat{x}-n)=2\eta\widehat{x}.
\end{equation}
The function $z\mapsto(z-n-1)(z-n)-2\eta z$, evaluated at $x+1$ and
$x+2$, takes the values
\[
 -2\eta
 \quad\text{and}\quad
 \frac{2n(d+1)}x ,
\]
by~\eqref{eq:section4-notation}. The polynomial is upward-opening and negative
at $x+1$, so $x+1$ lies strictly between its two roots; since it is positive at
$x+2$, its larger root therefore satisfies
\[
 x+1<\widehat{x}<x+2.
\]
Put
\[
 \delta=\widehat{x}-x,
\]
so that $1<\delta<2$. We improve the upper bound and show that
\begin{equation}\label{eq:eta-plus-delta}
 \eta+\delta<2.
\end{equation}
The two Saint-James relations $2\eta x=d(d+1)$ and~\eqref{eq:next-saint-james}
give, after subtraction,
\begin{equation}\label{eq:eta-delta}
 2\eta\delta=(\delta-1)(2d+\delta).
\end{equation}
Using $2\eta x=d(d+1)$ in~\eqref{eq:eta-delta}, we obtain
\[
 (\delta-1)x=\frac{d(d+1)\delta}{2d+\delta}
\]
and hence
\[
 d-(\delta-1)x=\frac{d^2(2-\delta)}{2d+\delta}>0.
\]
Together with~\eqref{eq:eta-x-less-n}, this yields
\[
 x(2-\delta-\eta)
 =(n-\eta x)+\bigl(d-(\delta-1)x\bigr)>0.
\]
Consequently,~\eqref{eq:eta-plus-delta} holds.

\subsection{The convexity estimate}\label{subsec:convexity}

Put $\nu=(1-\eta)/2$, $F=M(\nu,n+1,\cdot)$, and
\[
 Y(z)=2z\frac{F'(z)}{F(z)}-(z-n).
\]
The sector-$n$ Neumann condition~\eqref{eq:kummer-log-derivative} at the
crossing $\beta_{n-1}=2x$ gives $Y(x)=0$, while~\eqref{eq:R-as-Y} and the
definition of $\widehat x$ give
\begin{equation}\label{eq:R-Y-relation}
 \mathcal R_n(\eta)=Y(\widehat{x}).
\end{equation}
Here we have used, once more, that the parameter $\nu$ attached
by~\eqref{eq:x-nu} to the normalized energy $\eta$ is the same at the two fields
$2x$ and $2\widehat x$. Kummer's equation (\cite[Equation~13.2.1]{DLMF}) implies
\begin{equation}\label{eq:Y-riccati}
 Y'(z)=\frac{(z-n)^2-Y(z)^2}{2z}-\eta.
\end{equation}

\begin{lemma}\label{lem:Y-positive}
It holds that $Y(\widehat{x})>0$.
\end{lemma}

\begin{proof}
Assume, for contradiction, that $Y(\widehat{x})\leq0$. We show first that $Y$ is
then nonpositive on all of $[x,\widehat x]$, and then that this gives a convex
lower bound for $Y'$ with positive average, whence
$Y(\widehat x)>0$.

\emph{Step 1: $Y\leq0$ on $[x,\widehat x]$.} At a zero of $Y$,
equation~\eqref{eq:Y-riccati} gives
$Y'(z)=g(z)$, where
\[
 g(z)=\frac{(z-n)^2}{2z}-\eta .
\]
Since $g'(z)=(z-n)(z+n)/(2z^2)>0$ for $z>n$, the function $g$ is strictly
increasing on $(n,+\infty)$, and by~\eqref{eq:eta-x-less-n}
\[
 g(x)=\frac{d^2}{2x}-\eta=\frac{d^2-d(d+1)}{2x}=-\frac{d}{2x}<0 .
\]
Suppose $Y$ were positive somewhere in $(x,\widehat x)$, and let
$(\alpha,\gamma)$ be a connected component of $\{Y>0\}\cap(x,\widehat x)$. Since
$Y(x)=0$ and $Y(\widehat x)\leq0$, we have $Y(\alpha)=Y(\gamma)=0$ with
$x\leq\alpha<\gamma\leq\widehat x$. As $Y>0$ immediately to the right of $\alpha$
and immediately to the left of $\gamma$, we get $g(\alpha)=Y'(\alpha)\geq0$ and
$g(\gamma)=Y'(\gamma)\leq0$, contradicting the strict monotonicity of $g$. (If
$\alpha=x$, the contradiction is immediate from $g(x)<0$.) Hence
\begin{equation}\label{eq:Y-nonpositive}
 Y(z)\leq0, \qquad x\leq z\leq\widehat{x}.
\end{equation}

\emph{Step 2.} By Lemma~\ref{lem:euler}, $F>0$ and $F'>0$, so $Y(z)>-(z-n)$. It
follows from~\eqref{eq:Y-nonpositive} that, for $0<t\leq\delta$,
\[
 \frac d{dt}\bigl(Y(x+t)+\eta t\bigr)
 =\frac{(d+t)^2-Y(x+t)^2}{2(x+t)}>0 .
\]
Since $Y(x)=0$, we obtain
\[
 -\eta t<Y(x+t)\leq0,
 \qquad\text{hence}\qquad
 Y(x+t)^2<\eta^2t^2
 \quad (0<t\leq\delta).
\]
Integrating~\eqref{eq:Y-riccati} therefore gives
\begin{equation}\label{eq:Y-convex-lower}
 Y(\widehat{x})>
 \int_0^\delta K(t)\dd t,
 \qquad
 K(t)=\frac{(d+t)^2-\eta^2t^2}{2(x+t)}-\eta.
\end{equation}
Differentiating twice, and recalling that $d=x-n$, we get by~\eqref{eq:eta-x-less-n} that
\[
 K''(t)=\frac{n^2-\eta^2x^2}{(x+t)^3}>0.
\]
Thus, $K$ is convex, and
\[
 \int_0^\delta K(t)\dd t
 \geq\delta K\left(\frac\delta2\right).
\]
We show next that $K(\delta/2)>0$. Averaging the two Saint-James relations $2\eta
x=d(d+1)$ and~\eqref{eq:next-saint-james} gives
\[
 2\eta\left(x+\frac\delta2\right)
 =\frac{d(d+1)+(d+\delta-1)(d+\delta)}2 .
\]
Consequently,
\[
 2\left(x+\frac\delta2\right)K\left(\frac\delta2\right)
 =\left(d+\frac\delta2\right)^2
   -\frac{\eta^2\delta^2}{4}
   -2\eta\left(x+\frac\delta2\right)
 =\frac\delta4\bigl(2-\delta-\eta^2\delta\bigr).
\]
By~\eqref{eq:eta-plus-delta}, one has $\eta<2-\delta$. Since $1<\delta<2$,
\[
 \eta\delta<\delta(2-\delta)=1-(\delta-1)^2<1 .
\]
Thus $\eta^2\delta=\eta\cdot\eta\delta<\eta<2-\delta$, and therefore
$K(\delta/2)>0$. By~\eqref{eq:Y-convex-lower} this gives $Y(\widehat x)>0$,
contradicting the assumption $Y(\widehat{x})\leq0$.
\end{proof}

\begin{proof}[Proof of Theorem~\ref{thm:crossing-quotients}]
Fix $n\geq1$. By~\eqref{eq:R-Y-relation} and Lemma~\ref{lem:Y-positive},
$\mathcal R_n(\eta)>0$, while $\mathcal R_n(1)<0$ by~\eqref{eq:R-at-one}. Since
$\mathcal R_n$ is continuous on $(0,1]$, it has a zero in $(\eta,1)$. By
Lemma~\ref{lem:zeros-of-R} that zero is $\eta_n^*$, whence
$\eta_{n-1}^*=\eta<\eta_n^*$.
\end{proof}

\subsection{The de Gennes bound}\label{subsec:de-gennes-from-crossings}

\begin{proof}[First proof of Theorem~\ref{thm:de-gennes-bound}]
Since $\beta_n>2(n+1)$, one has $\beta_n\to\infty$, so it follows
from~\eqref{eq:de-gennes-limit} that
\[
 \eta_n^*=\frac{\lambda(\beta_n)}{\beta_n} \to \Theta_0 .
\]
Theorem~\ref{thm:crossing-quotients} therefore gives
\[
 \eta_n^*<\Theta_0 \qquad(n\geq0).
\]
Let $n\geq1$. By (HL2) the function $b\mapsto\lambda_n(b)/b$ decreases to a
unique minimum and then increases, so it has no interior maximum on
$[\beta_{n-1},\beta_n]$ and attains
its maximum there at an endpoint. Since $\lambda_n(\beta_{n-1})/\beta_{n-1}
=\eta_{n-1}^*$ and $\lambda_n(\beta_n)/\beta_n=\eta_n^*$, we obtain, for
$\beta_{n-1}\leq b\leq\beta_n$,
\[
 \frac{\lambda(b)}b
 =\frac{\lambda_n(b)}b
 \leq\max\{\eta_{n-1}^*,\eta_n^*\}
 =\eta_n^*<\Theta_0 .
\]
On the initial interval, the monotonicity of $\lambda_0(b)/b$ from (HL2) gives
\[
 \frac{\lambda(b)}b\leq\eta_0^*<\Theta_0
 \qquad(0<b\leq\beta_0).
\]
This proves Theorem~\ref{thm:de-gennes-bound}.
\end{proof}

\section{A second, variational proof of the de Gennes bound}
\label{sec:trial-state-proof}

In this section we give a second proof of Theorem~\ref{thm:de-gennes-bound},
and prove Theorem~\ref{thm:quantitative}. The argument is independent of
Sections~\ref{sec:preliminaries}--\ref{sec:crossing-monotonicity}, and of
Theorem~\ref{thm:HL}; in particular, it does not use the strong-field
asymptotics~\eqref{eq:de-gennes-limit}. It is purely variational: we construct,
for each $b\in(0,+\infty)$, a trial function $v_b\in H^1(\D)$ such that
\begin{equation}\label{eq:var-upper}
 \mathcal{Q}_b[v_b] \coloneqq q_b[v_b]-\Theta_0 b\norm{v_b}^2<0 .
\end{equation}
The variational characterization of $\lambda(b)$ then implies
$\lambda(b)<\Theta_0 b$.

The construction is divided into three regimes. Taking inspiration from the
asymptotic analysis of $\lambda(b)$ for large $b$ (see for
instance~\cite{MR2662319}), we build in Section~\ref{subsec:large} a trial
function from the de Gennes ground state $t\mapsto\varphi_0(t)$, and show that it
satisfies~\eqref{eq:var-upper} for all $b\geq130$. For small $b$ a constant
function suffices (Section~\ref{subsec:small}). On the remaining bounded interval
we use trial states from a finite-dimensional space of polynomial functions, and
verify~\eqref{eq:var-upper} on finitely many overlapping intervals by exact
rational computations (Section~\ref{subsec:intermediate}). Throughout we use the
certified numerical bounds
\begin{equation}\label{eq:numerical-constants}
  \abs {\Theta _ 0 -  \num {0.590106125}} \leq 10 ^ {-9},\quad
  \abs {\varphi_0(0) - \num {0.8730}} \leq 10 ^ {-4},
\end{equation}
taken from~\cite[Theorem 1.1]{MR2912745}. In particular
$\tfrac12<\Theta_0<1$.

\subsection{Large magnetic field}\label{subsec:large}

Recall that $\xi_0 > 0$ denotes the minimizer of $\xi\mapsto \mu(\xi)$ in the
definition of $\Theta_0$ in~\eqref{eq:Theta0}, and that $\varphi_0$ is the
corresponding positive normalized eigenfunction of $-d^2/dt^2+(t-\xi_0)^2$ on
$\R_+$ with Neumann boundary condition. At the parameter $\xi=\xi_0$, the
vanishing of $\mu'(\xi_0)$ gives $\Theta_0=\xi_0^2$. We
consider a trial state of the form
\begin{equation}\label{eq:trial-large}
 u_{m,b}(r,\theta)
 =\varphi_0\bigl(\sqrt b\log(1/r)\bigr)
  \frac{e^{-im\theta}}{\sqrt{2\pi}},
 \qquad m\in\N_0.
\end{equation}
The integer parameter $m$ can be interpreted as the angular momentum. A key point
of our construction lies in choosing a suitable value for $m$ as a function of
$b$. To further motivate the definition~\eqref{eq:trial-large}, we point out that
the change of variables $t=\sqrt b\log(1/r)$ is a smooth diffeomorphism mapping
the interval $(0,1]$ to $[0,+\infty)$ in such a way that $t=0$ corresponds to
$r=1$, that is, to the boundary of the domain $\D$.
Moreover, as $r\approx 1$ we have $t \approx \sqrt {b}(1-r)$, so $t$ is locally
proportional to the distance to the boundary. This is the scale that appears in
the asymptotic theory; see~\cite[Section~8.2]{MR2662319}.
Since $\varphi_0$ decays
faster than any power at infinity, $u_{m,b}\in H^1(\D)$. A direct computation,
using the differential equation satisfied by $\varphi_0$, the identity
$\int_0^{+\infty}\varphi_0'(t)^2\dd t=\Theta_0/2$ (see the proof of Lemma~3.2.7
in~\cite{MR2662319}), and $r\dd r = b^{-1/2}e^{-2t/\sqrt{b}}\dd t$, yields
\[
  \begin{aligned}
    q_{b}[u_{m,b}]
    &
    =
    \frac {\sqrt {b}\xi _ 0 ^ 2}{2}
    + \frac {m ^ 2}{\sqrt {b}}
    + \int _ 0 ^ {+\infty}
      \Bigl[
        \frac {b ^ {3/2}}{4} e ^ {-4t/\sqrt {b}} \varphi_0(t) ^ 2
        - m\sqrt {b} e ^ {-2t/\sqrt {b}} \varphi_0(t) ^ 2
      \Bigr] \dd t ,
      \\
      \norm {u_{m,b}} ^ 2
  &
  =
  \int _ 0 ^ {+\infty} \frac {1}{\sqrt {b}} e ^ {-2t/\sqrt {b}} \varphi_0(t) ^ 2 \dd t.
  \end{aligned}
\]
We use the elementary estimate
\[
  1 - \tau  + \frac {\tau ^ 2}{2} - \frac {\tau ^ 3}{6}
  <
  e ^ {-\tau}
  <
  1 - \tau  + \frac {\tau ^ 2}{2} - \frac {\tau ^ 3}{6} + \frac {\tau ^ 4}{24}
  \quad \text{for all } \tau > 0
\]
to bound the exponential functions by polynomials in $t$, the lower bound being
used for the two terms that enter with a negative sign. We obtain the upper
bound
\[
  \begin{aligned}
    \mathcal{Q}_b[u_{m,b}]
    &
    \leq
    \frac {b ^ {1/2}\xi _ 0 ^ 2}{2}
    + m ^ 2 b ^ {-1/2}
    + \frac {1}{4} b ^ {3/2}
    - b T _ 1
    + 2 b ^ {1/2} T _ 2
    - \frac {8}{3} T _ 3
    + \frac {8}{3} b ^ {-1/2} T _ 4
    \\
    & \qquad
    - m b ^ {1/2}
    + 2 m T _ 1
    - 2 m b ^ {-1/2} T _ 2
    + \frac {4}{3} m b ^ {-1} T _ 3
    \\
    & \qquad
    - \xi _ 0 ^ 2 b
    \Bigl(
      b ^ {-1/2}
      - 2 b ^ {-1} T _ 1
      + 2 b ^ {-3/2} T _ 2
      - \frac {4}{3} b ^ {-2} T _ 3
    \Bigr),
  \end{aligned}
\]
where $T_k$ denotes the $k$-th moment of $\varphi_0^2$, defined by
$T_k \coloneqq \int_0^{+\infty} t^k\varphi_0(t)^2\dd t$. The right-hand side is a
quadratic polynomial in the angular momentum \( m \), which attains its minimum
at \( m = m _ {\mathrm {opt}}\), where
\[
  m _ {\mathrm {opt}}
  =
  \frac {b}{2} - T _ 1 b ^ {1/2} + T _ 2 - \frac {2}{3} T _ 3 b ^ {-1/2}.
\]
Since \( m \) has to be an integer, we cannot always use this value, but we can
choose an integer \( \underline m(b) = m _ {\mathrm {opt}} + \epsilon \), with
\( \abs {\epsilon} \leq 1/2\) ($\underline m(b)$ is unique except when
$m _ {\mathrm {opt}}$ belongs to $1/2+\N_0$). We now assume
$m=\underline m(b)$ and insert the values of the moments
\[
  T _ 1 = \xi _ 0,\quad
  T _ 2 = \frac {3}{2}\xi _ 0 ^ 2,\quad
  T _ 3 = \frac {C _ 1}{2} + \frac {5}{2} \xi _ 0 ^3,\quad
  T _ 4 = \frac {3}{8} + \frac {35}{8}\xi _ 0 ^ 4 + \frac {7}{8} C _ 1 \xi _ 0,
\]
computed using~\cite[Lemma~3.2.7 and Equation~(3.54)]{MR2662319}, with the
constant $C_1 = \varphi_0(0)^2/3$. Note that, for $b\geq4\Theta_0=4\xi_0^2$, one
has $b/2-\xi_0\sqrt b\geq0$ and $\frac23T_3b^{-1/2}\leq T_3/(3\xi_0)$,
whence
\[
 m_{\mathrm{opt}}\geq\frac23\xi_0^2-\frac{C_1}{6\xi_0}>0
\]
because $4\xi_0^3>C_1$ by~\eqref{eq:numerical-constants}; in particular
$\underline m(b)\in\N_0$, as required in~\eqref{eq:trial-large}. We discard the
term in $b^{-3/2}$ on the right-hand side, which appears with a negative
coefficient. We finally obtain
\begin{equation}\label{eq:var-upper-2}
\begin{aligned}
  \mathcal{Q}_b[u_{\underline m(b),b}]
  &\leq
  - C _ 1
  + \Bigl(\frac {5}{4} + \frac {5}{3} C _ 1 \xi _ 0
  + \frac {37}{12} \xi _ 0 ^ 4\Bigr) b ^ {-1/2}
  + \frac {5}{3} \xi _ 0 ^ 2\bigl(C _ 1 + 5 \xi _ 0 ^ 3\bigr)b ^ {-1}
  \\
  &\eqqcolon -C_1 + B_1 b ^ {-1/2} + B_2 b ^ {-1},
  \end{aligned}
\end{equation}
where $C_1$, $B_1$ and $B_2$ are the positive constants of
Theorem~\ref{thm:quantitative}. The right-hand side of~\eqref{eq:var-upper-2} is
negative if \( b > b _ 0 \), with
\begin{equation}\label{eq:b0}
  b _ 0 \coloneqq \frac {\bigl(B_1 + \sqrt {B_1 ^ 2 + 4C_1B_2}\bigr) ^ 2} {4C_1 ^ 2}.
\end{equation}
Using~\eqref{eq:numerical-constants} and exact rational computations, we deduce
the upper bound $b_0<130$. We have thus proved the following.
\begin{lemma}\label{lem:largeB}
  If \( b \ge 130 \), then \(\lambda (b) < \Theta _ 0 b\).
\end{lemma}

\subsection{Small magnetic field}\label{subsec:small}

Using a constant function as a trial state, we find that
\( \lambda (b) \leq b ^ 2/8\), and so \(\lambda (b) < \Theta _ 0 b\)
for \( 0 < b < 8\Theta _ 0 \). Since \(\Theta _ 0 > 1/2\)
by~\eqref{eq:numerical-constants}, we have the following.
\begin{lemma}\label{lem:smallB}
  If \(0 < b \le 4 \), then \(\lambda (b) < \Theta _ 0 b\).
\end{lemma}

\subsection{Intermediate magnetic field}\label{subsec:intermediate}

We use a different approach for the remaining range $4<b<130$; the construction
below covers, with some margin, the larger interval $[3,131]$. For a fixed
angular momentum $m\geq1$, we define the vector space of trial functions
\[
\mathcal P_m \coloneqq  \biggl\{
 r^m\sum_{j=0}^8 c_j (1-r^2)^j \frac{e^{-im\theta}}{\sqrt{2\pi}}
 :c_0,\ldots,c_8\in\R
 \biggr\}\subset H^1(\D).
\]
We parametrize the functions in $\mathcal P_m$ by vectors
$c=(c_0,\dots,c_8)\in \R^9$, using the notation
\[
  u_c(r\cos(\theta),r\sin(\theta))
  \coloneqq \sum_{j=0}^8 c _ j \psi _ j (r)\frac{e^{-im\theta}}{\sqrt{2\pi}},
\]
where $\psi _ j (r) = r ^ m (1 - r ^ 2) ^ j$. Denoting by
$\innerproduct{\cdot,\cdot}$ the inner product of $L^2((0,1),r\dd r)$, and by
$\innerproduct{\cdot,\cdot}_{\R^9}$ the Euclidean one, there exist real symmetric
positive-definite $9\times9$ matrices $\mathsf{M}$, $\mathsf{K}$ and $\mathsf{V}$
(which can be interpreted as mass, kinetic, and potential matrices, respectively)
such that
\begin{equation}\label{eq:matrix-forms}
  \norm {u_c} ^ 2 = \innerproduct{c,\mathsf{M}c}_{\R^9},\quad
  q_b[u_c] = \innerproduct{c,\mathsf{K}c}_{\R^9}
  - mb\innerproduct{c,\mathsf{M}c}_{\R^9}
  + b ^ 2 \innerproduct{c,\mathsf{V}c}_{\R^9}.
\end{equation}
Explicitly,
\[
    \mathsf{M} _ {j,k}
    = \innerproduct {\psi _ j,\psi _ k},\qquad
    \mathsf{K} _ {j,k}
    = \innerproduct {\psi _ j ',\psi _ k '}
        + \innerproduct {(m^2/r^2) \psi _ j,\psi _ k},\qquad
    \mathsf{V} _ {j,k}
    =  \innerproduct {(r^2/4) \psi _ j,\psi _ k}.
\]
These matrix entries can be computed explicitly using the Euler beta function,
and are all rational numbers.

We now define, for $u\in H^1(\D)$,
\[
  \mathcal{Q}_b^*[u]\coloneqq q_b[u]-\Theta_*b\norm{u}^2,
\]
where $\Theta_*=5901/10000<\Theta_0$ by~\eqref{eq:numerical-constants}. As
above, given any $b>0$, if there exists a trial function $v\in H^1(\D)$ such
that $\mathcal{Q}_b^*[v]<0$, it follows from the variational characterization of
$\lambda(b)$ that $\lambda(b)<\Theta_*b<\Theta_0b$.

Given a coefficient vector $c\in \R^9$, it follows from~\eqref{eq:matrix-forms}
and from the properties of the matrices $\mathsf{M}$, $\mathsf{K}$ and
$\mathsf{V}$ that
\[
 p_{m,c}(b) \coloneqq \mathcal{Q}_b^*[u_c]=q_b[u_c]-\Theta_*b\norm{u_c}^2
\]
is a convex quadratic polynomial in $b$, its leading coefficient being
$\innerproduct{c,\mathsf{V}c}_{\R^9}>0$. The whole of the certification rests on the
following elementary observation.

\begin{lemma}[Certification]\label{lem:certification}
Let $m\geq1$ be an integer and let $c\in\Q^9\setminus\{0\}$. If $b_{\min}$ and
$b_{\max}$ are positive integers with $b_{\min}<b_{\max}$,
$p_{m,c}(b_{\min})<0$ and $p_{m,c}(b_{\max})<0$, then
\[
 \lambda(b)<\Theta_*b<\Theta_0b
 \qquad\text{for every } b\in[b_{\min},b_{\max}].
\]
Moreover $p_{m,c}$ has rational coefficients, so that
$p_{m,c}(b_{\min})$ and $p_{m,c}(b_{\max})$ are rational numbers and the two
hypotheses can be checked exactly by rational arithmetic.
\end{lemma}

\begin{proof}
The polynomial $p_{m,c}$ is convex, hence negative on $[b_{\min},b_{\max}]$ as
soon as it is negative at the two endpoints; the conclusion follows from the
variational characterization of $\lambda(b)$ recalled above. The coefficients of
$p_{m,c}$ are rational because the entries of $\mathsf{M}$, $\mathsf{K}$ and
$\mathsf{V}$ are rational, $c\in\Q^9$ and $\Theta_*\in\Q$.
\end{proof}

Our strategy therefore is to construct a finite family of intervals
$[b_{\min},b_{\max}]$ with integer endpoints, each with its corresponding
rational vector $c$ satisfying the hypotheses of
Lemma~\ref{lem:certification}, in such a way that the union of the intervals
covers $[3,131]$.

To generate such a family we proceed as follows, with all the computations done
in Mathematica. For each $m=1,\ldots,56$, we define
$b _ {\mathrm {ini}} \coloneqq 2m + 2.25\sqrt{m}$ --- a choice motivated by
Remark~\ref{rem:beta-asymptotics} --- and compute a vector $c\in \R^9$ that
minimizes the Rayleigh quotient
\[
	\frac{q_{b _ {\mathrm {ini}}}[u_c]}{\norm{u_c}^2}
  =
  \frac
  {\innerproduct{c,\mathsf{K}c}_{\R^9}
  - mb _ {\mathrm {ini}}\innerproduct{c,\mathsf{M}c}_{\R^9}
  + b _ {\mathrm {ini}} ^ 2 \innerproduct{c,\mathsf{V}c}_{\R^9}}
  {\innerproduct{c,\mathsf{M}c}_{\R^9}}.
\]
This is done by computing an eigenvector associated with the lowest eigenvalue of
the matrix $\mathsf{K}-mb _ {\mathrm {ini}}\mathsf{M}+b _ {\mathrm {ini}}^2\mathsf{V}$ relative to the
matrix $\mathsf{M}$, and normalizing by the condition $c_0=1$. We then find the roots
$b_-<b_+$ of the polynomial $p_{m,c}$. All the computations up to this point are
performed in floating-point arithmetic; they serve only to produce candidate data
and play no role in the proof. We finally define $b_{\min}^{m}$ as the smallest
integer greater than or equal to $b_-$, $b_{\max}^{m}$ as the largest integer
less than or equal to $b_+$, and $c_m$ as the rationalization of $c$ with
granularity $10^{-1}$. It is not clear a priori that
$0<b_{\min}^{m}<b_{\max}^{m}$, nor that $p_{m,c_m}$ is negative at those two
points. The parameters in this construction, namely the range of values for the
angular momentum \(m\), the dimension of the vector space $\mathcal P_m$ and the
granularity of the rationalization, had to be adjusted by trial and error to make
these properties hold. The procedure yields a family of intervals
\[
  \left\{[b_{\min}^m,b_{\max}^m] : m =1,\dots,56 \right\}
\]
which overlap (or abut) and cover $[3,131]$, and each of which satisfies the
hypotheses of Lemma~\ref{lem:certification} with the vector $c_m$. In fact, one
can construct a smaller family of intervals that covers $[3,131]$ by retaining
only some values of $m$ (see Appendix~\ref{sec:details} for details). This yields
the following result.

\begin{lemma}\label{lem:intermediateB}
  If \(  3 \le b \le 131 \), then \(\lambda (b) < \Theta _ 0 b\).
\end{lemma}

It is worth noting that validating Lemma~\ref{lem:intermediateB} requires only a
list of intervals with integer endpoints and corresponding rational vectors,
together with a finite number of exact rational computations, by
Lemma~\ref{lem:certification}, whatever the method used to produce the list. We
give our list in Tables~\ref{tab:intervaltable} and~\ref{tab:mandvec} in
Appendix~\ref{sec:details}. For completeness, we also include the Mathematica
code which implements the procedure described above.

\subsection{Proof of Theorems~\ref{thm:de-gennes-bound}
and~\ref{thm:quantitative}}

Lemma~\ref{lem:largeB}, Lemma~\ref{lem:smallB} and
Lemma~\ref{lem:intermediateB} cover, respectively, $[130,+\infty)$, $(0,4]$ and
$[3,131]$, whose union is $(0,+\infty)$. Together they therefore give a second
proof of Theorem~\ref{thm:de-gennes-bound}.

\begin{proof}[Proof of Theorem~\ref{thm:quantitative}]
For the trial
state~\eqref{eq:trial-large} we have, on the one hand, $e^{-2t/\sqrt b}<1$ and
hence
\[
 \norm {u_{m,b}} ^ 2
 =\int _ 0 ^ {+\infty} \frac {1}{\sqrt {b}} e ^ {-2t/\sqrt {b}} \varphi_0(t) ^ 2 \dd t
 <\frac1{\sqrt b},
\]
and, on the other hand, by $e^{-\tau}>1-\tau$,
\[\begin{aligned}
\norm {u_{m,b}} ^ 2
   & >
  \int _ 0 ^ {+\infty} \frac {1}{\sqrt {b}}\left(1-\frac{2t}{\sqrt b}\right) \varphi_0(t) ^ 2 \dd t
   = \frac{1}{\sqrt b}-\frac{2}{b}T_1=\frac{1}{\sqrt{b}}\left(1-\frac{2\xi_0}{\sqrt b}\right).
  \end{aligned}
   \]
Let $b>4\Theta_0$, so that $1-2\xi_0b^{-1/2}>0$, and write
$u=u_{\underline m(b),b}$. Then
\[
 \lambda(b)\leq\frac{q_b[u]}{\norm u^2}
 =\Theta_0b+\frac{\mathcal{Q}_b[u]}{\norm u^2}
 \leq\Theta_0b+\frac{-C_1}{\norm u^2}+\frac{B_1b^{-1/2}+B_2b^{-1}}{\norm u^2},
\]
by~\eqref{eq:var-upper-2}. In the second term the numerator is negative, so we
insert the \emph{upper} bound for $\norm u^2$; in the third term the numerator
is positive, so we insert the \emph{lower} bound. This gives
\[
 \lambda(b)
 <\Theta_0b-C_1\sqrt b
 +\frac{B_1+B_2b^{-1/2}}{1-2\xi_0b^{-1/2}},
\]
which is the first assertion of Theorem~\ref{thm:quantitative}. If moreover
$b\geq130>b_0$, then $-C_1+B_1b^{-1/2}+B_2b^{-1}<0$ by~\eqref{eq:b0}, and the
upper bound $\norm u^2<b^{-1/2}$ applied to the whole of $\mathcal{Q}_b[u]$ gives
directly
\[
 \lambda(b)<\Theta_0b+\sqrt b\bigl(-C_1+B_1b^{-1/2}+B_2b^{-1}\bigr)
 =\Theta_0b-C_1\sqrt b+B_1+B_2b^{-1/2},
\]
which is the second assertion.
\end{proof}

\section*{Acknowledgments}

C.~L. acknowledges support from the INdAM GNAMPA Project \emph{Functional and
spectral analysis for differential operators} (CUP E53C25002010001) and thanks
the Isaac Newton Institute for Mathematical Sciences, Cambridge, for support
and hospitality during the programme \emph{Geometric spectral theory and
applications} (EPSRC grant EP/Z000580/1), where part of this work was carried
out.

Both authors would like to thank Søren Fournais, Bernard Helffer, Ayman Kachmar
and Germán Miranda for fruitful discussions. OpenAI's Codex tool gave valuable
suggestions regarding the English language and the organization of the text.

\appendix

\section{Program and tables}\label{sec:details}

Here we collect the Mathematica code used for the calculations and the tables
containing their results.

\begin{verbatim}
terms = 8;
granularity = 10^(-1);

matM[m_] =
  Table[
    1/2 Beta[m + 1, j + k + 1],
    {j, 0, terms},{k, 0, terms}
  ];

matK[m_] =
  Table[
    1/2(2 m^2 Beta[m, j + k + 1]
    - 2m (j + k) If[j + k > 0, Beta[m + 1, j + k], 0]
    + 4j k If[j + k > 1, Beta[m + 2, j + k - 1], 0]),
    {j, 0, terms}, {k, 0, terms}
  ];

matV[m_] =
  Table[
    1/8 Beta[m + 2, j + k + 1],
    {j, 0, terms}, {k, 0, terms}
  ];

smalleig[b_, m_] :=
  Eigensystem[
    {N[matK[m] - m b matM[m] + b^2 matV[m]], N[matM[m]]}, -1
  ];

findInterval[m_] :=
  Module[
    {eigs, vec, pol, btab, bleft, bright,
     rvec, leftval, rightval},
    eigs = smalleig[2*m + 2.25*Sqrt[m], m];
    vec = eigs[[2,1]]/eigs[[2,1,1]];
    pol = Collect[
            Dot[
             {{vec}}.(matK[m] - (m + 5901/10000) b matM[m]
                                 + b^2 matV[m]),
             vec
             ], b][[1,1]];
    btab     = SolveValues[pol==0,b];
    bleft    = Ceiling[btab[[1]]];
    bright   = Floor[btab[[2]]];
    rvec     = Rationalize[vec, granularity];
    leftval  = Dot[
                {{rvec}}.(matK[m]
                  - (m + 5901/10000) bleft matM[m]
                  + bleft^2 matV[m] ),
                rvec
                ][[1,1]];
    rightval = Dot[
                {{rvec}}.(matK[m]
                  - (m + 5901/10000) bright matM[m]
                  + bright^2 matV[m] ),
                rvec
                ][[1,1]];
    {m, bleft, bright, leftval, rightval, vec,rvec}
  ];

Table[findInterval[m], {m, 1, 56}]//TableForm
\end{verbatim}

After running the previous code to generate the intervals corresponding to each
integer $m$ between $1$ and $56$, we manually selected a subfamily of them that
still overlaps and covers $[3,131]$, as illustrated by the following table.

\begin{longtblr}[
  caption={Table of the overlapping intervals and the corresponding angular momentum
  \(m\). The endpoints are the \texttt {bleft} and \texttt {bright} in the Mathematica
  output.},
  label=tab:intervaltable]
  {colspec={lp{2cm}lp{2cm}lp{2cm}ll}}
  \toprule
  \( m\) & \([b _ {\mathrm {left}},b _ {\mathrm {right}}]\) &
  \( m\) & \([b _ {\mathrm {left}},b _ {\mathrm {right}}]\) &
  \( m\) & \([b _ {\mathrm {left}},b _ {\mathrm {right}}]\) &
  \( m\) & \([b _ {\mathrm {left}},b _ {\mathrm {right}}]\) \\
  \midrule
 \( 1\) & \([ 3, 7]\) &  \(13\) & \([32,37]\) & \(29\) & \([ 68, 73]\) & \(46\) & \([105,110]\) \\
 \( 2\) & \([ 6,10]\) &  \(15\) & \([37,42]\) & \(31\) & \([ 72, 78]\) & \(48\) & \([109,114]\) \\
 \( 3\) & \([ 9,13]\) &  \(17\) & \([41,46]\) & \(33\) & \([ 76, 82]\) & \(50\) & \([113,119]\) \\
 \( 4\) & \([11,15]\) &  \(19\) & \([46,51]\) & \(35\) & \([ 81, 86]\) & \(52\) & \([118,123]\) \\
 \( 5\) & \([14,18]\) &  \(21\) & \([50,55]\) & \(37\) & \([ 85, 91]\) & \(54\) & \([122,127]\) \\
 \( 7\) & \([18,23]\) &  \(23\) & \([54,60]\) & \(39\) & \([ 89, 95]\) & \(56\) & \([126,131]\) \\
 \( 9\) & \([23,28]\) &  \(25\) & \([59,64]\) & \(41\) & \([ 94,100]\) &        &  \\
 \(11\) & \([28,32]\) &  \(27\) & \([63,69]\) & \(44\) & \([100,106]\) &        &  \\
  \bottomrule
\end{longtblr}

\begin{longtblr}[
  caption={The rationalized coefficient vectors used for each relevant \(m\), called
  \texttt {rvec} in the Mathematica code.},
  label=tab:mandvec]{colspec={Q[l,$]Q[l,$]}}
  \toprule
    m & c _ {\mathrm {rat}} \\
  \midrule
  1 & (1 , \frac12 , \frac13 , 0 , 0 , 0 , 0 , 0 , 0) \\ 2 & (1 , 1 , \frac45 , \frac12 , \frac15 , 0 , 0 , 0 , 0) \\ 3 & (1 , \frac32 , \frac53 , \frac54 , \frac23 , \frac13 , \frac15 , 0 , 0) \\ 4 & (1 , 2 , \frac{11}{4} , \frac83 , 2 , 1 , \frac45 , -\frac18 , \frac13) \\ 5 & (1 , \frac52 , \frac{21}{5} , \frac{14}{3} , \frac{13}{3} , \frac73 , \frac{10}{3} , -\frac87 , \frac53) \\ 7 & (1 , \frac72 , \frac{23}{3} , \frac{34}{3} , 15 , \frac92 , \frac{218}{7} , -\frac{68}{3} , 23) \\ 9 & (1 , \frac92 , \frac{61}{5} , \frac{65}{3} , \frac{128}{3} , -\frac{57}{4} , 186 , -\frac{572}{3} , \frac{695}{4}) \\ 11 & (1 , \frac{11}{2} , \frac{71}{4} , \frac{71}{2} , \frac{215}{2} , -\frac{273}{2} , 814 , -\frac{5109}{5} , \frac{1809}{2}) \\ 13 & (1 , \frac{13}{2} , \frac{49}{2} , \frac{155}{3} , 250 , -\frac{1742}{3} , \frac{14189}{5} , -\frac{12286}{3} , 3636) \\ 15 & (1 , \frac{15}{2} , \frac{65}{2} , 67 , \frac{2161}{4} , -\frac{5438}{3} , \frac{24976}{3} , -13371 , \frac{36305}{3}) \\ 17 & (1 , \frac{17}{2} , \frac{167}{4} , 77 , \frac{3269}{3} , -\frac{9385}{2} , 21356 , -\frac{149829}{4} , \frac{69719}{2}) \\ 19 & (1 , \frac{19}{2} , \frac{158}{3} , \frac{224}{3} , 2063 , -\frac{31963}{3} , 49226 , -93158 , 89568) \\ 21 & (1 , \frac{21}{2} , \frac{261}{4} , \frac{103}{2} , \frac{18439}{5} , -\frac{43819}{2} , \frac{416075}{4} , -\frac{842511}{4} , \frac{838987}{4}) \\ 23 & (1 , \frac{23}{2} , \frac{239}{3} , -\frac72 , \frac{18793}{3} , -\frac{166675}{4} , \frac{409097}{2} , -440457 , \frac{1819419}{4}) \\ 25 & (1 , \frac{25}{2} , \frac{673}{7} , -\frac{308}{3} , \frac{40687}{4} , -\frac{223117}{3} , \frac{3029353}{8} , -\frac{7766767}{9} , 924612) \\ 27 & (1 , \frac{40}{3} , \frac{574}{5} , -\frac{782}{3} , \frac{31749}{2} , -\frac{377767}{3} , \frac{1997789}{3} , -\frac{3200043}{2} , \frac{12449991}{7}) \\ 29 & (1 , \frac{43}{3} , 136 , -\frac{1971}{4} , \frac{95695}{4} , -\frac{611732}{3} , \frac{7843816}{7} , -\frac{5658829}{2} , 3261986) \\ 31 & (1 , \frac{46}{3} , \frac{639}{4} , -\frac{4081}{5} , 34961 , -\frac{1271349}{4} , \frac{3629299}{2} , -\frac{9604509}{2} , \frac{17217287}{3}) \\ 33 & (1 , \frac{49}{3} , \frac{559}{3} , -\frac{3745}{3} , \frac{149146}{3} , -\frac{1437943}{3} , \frac{14208964}{5} , -\frac{47179361}{6} , \frac{29204708}{3}) \\ 35 & (1 , \frac{86}{5} , \frac{647}{3} , -\frac{5423}{3} , \frac{138007}{2} , -702255 , \frac{12961936}{3} , -\frac{49895929}{4} , \frac{79933081}{5}) \\ 37 & (1 , \frac{127}{7} , 248 , -\frac{7538}{3} , \frac{281180}{3} , -\frac{6018269}{6} , \frac{19196624}{3} , -19239668 , 25505873) \\ 39 & (1 , 19 , \frac{567}{2} , -\frac{20293}{6} , \frac{249733}{2} , -1400695 , \frac{27771502}{3} , -\frac{115754659}{4} , \frac{158611557}{4}) \\ 41 & (1 , 20 , \frac{1289}{4} , -4435 , \frac{817406}{5} , -\frac{9584901}{5} , \frac{39340096}{3} , -\frac{255333389}{6} , \frac{240894767}{4}) \\ 44 & (1 , \frac{64}{3} , \frac{773}{2} , -\frac{25593}{4} , \frac{713732}{3} , -\frac{8905816}{3} , \frac{64044575}{3} , -\frac{219261415}{3} , 108403952) \\ 46 & (1 , \frac{67}{3} , \frac{1300}{3} , -\frac{15979}{2} , \frac{1201225}{4} , -\frac{19479664}{5} , \frac{173607013}{6} , -\frac{717470237}{7} , 156709469) \\ 48 & (1 , \frac{116}{5} , \frac{1451}{3} , -\frac{29482}{3} , \frac{1123261}{3} , -\frac{15124499}{3} , \frac{154542655}{4} , -\frac{848673709}{6} , \frac{1113821099}{5}) \\ 50 & (1 , \frac{217}{9} , \frac{1075}{2} , -\frac{23857}{2} , \frac{923203}{2} , -6440835 , \frac{203561117}{4} , -192346048 , \frac{935443570}{3}) \\ 52 & (1 , 25 , \frac{1189}{2} , -\frac{42928}{3} , \frac{1689754}{3} , -\frac{24398269}{3} , \frac{397190041}{6} , -\frac{516106957}{2} , \frac{2581853365}{6}) \\ 54 & (1 , 26 , \frac{3931}{6} , -16986 , 680821 , -10159438 , \frac{255377579}{3} , -\frac{1025779505}{3} , \frac{2344398115}{4}) \\ 56 & (1 , \frac{107}{4} , \frac{6472}{9} , -\frac{59920}{3} , \frac{3263233}{4} , -\frac{25132953}{2} , \frac{324928729}{3} , -\frac{1343653420}{3} , 788666174) \\
  \bottomrule
\end{longtblr}

\begin{longtblr}[
  caption={The exact values of \(q_b[u_c]-\Theta_*b\norm{u_c}^2\) at \(b _
  {\mathrm {left}}\) and \(b _ {\mathrm {right}}\), for each relevant \(m\). Note
  that each entry in the last two columns is negative. These are the \texttt
  {leftval} and \texttt {rightval} in the Mathematica code.},
  label=tab:results]{colspec={Q[l,$]Q[l,$]Q[l,$]}}
  \toprule
  m & \text {value at }b _ {\mathrm {left}} & \text {value at }b _ {\mathrm {right}} \\
  \midrule
1 & -\frac{880729}{302400000} & -\frac{6980243}{43200000} \\
2 & -\frac{2144023}{26400000} & -\frac{1869503}{15840000} \\
3 & -\frac{73520892394069}{441080640000000} & -\frac{7094026303811}{147026880000000} \\
4 & -\frac{4971090856187}{70424640000000} & -\frac{15003448229701}{88216128000000} \\
5 & -\frac{4249053120354893}{24094029960000000} & -\frac{7754229321162943}{90815959080000000} \\
7 & -\frac{229901300678299}{8586236131200000} & -\frac{572212216127921}{8212921516800000} \\
9 & -\frac{18398796418230253}{192752239680000000} & -\frac{1551290674032577}{48188059920000000} \\
11 & -\frac{84592882004325061}{517575458400000000} & -\frac{2500623465004013}{14930061300000000} \\
13 & -\frac{2578732715722848241}{40613499251325000000} & -\frac{147427764565925159917}{1299631976042400000000} \\
15 & -\frac{30875440105493333003}{207941116166784000000} & -\frac{407095654671444209}{9451868916672000000} \\
17 & -\frac{3724661946690179383}{59411747476224000000} & -\frac{30698078457823837943}{207941116166784000000} \\
19 & -\frac{651540067114299163}{4218103781892000000} & -\frac{163015704928402063}{2300783881032000000} \\
21 & -\frac{2244364675513502287}{27576420423552000000} & -\frac{1348606058510812141867}{8548690331301120000000} \\ 23 & -\frac{23967067879630140497}{13715072748913536000000} & -\frac{102969599903876229833}{1371507274891353600000} \\
25 & -\frac{3863002566132888480750743}{34677189938352984422400000} & -\frac{22443283454681057796421}{147772116214572376800000} \\
27 & -\frac{140699100781274945456741}{3440197414519145280000000} & -\frac{1592302659119514991451681}{24081381901634016960000000} \\
29 & -\frac{471140417993004054666013}{3264197346165062361600000} & -\frac{6520707489631430808108989}{47874894410420914636800000} \\
31 & -\frac{251030553812494478387}{2872677515986278000000} & -\frac{1328744570222466022531}{34472130191835336000000} \\
33 & -\frac{706637667383371004011}{27214839625133160000000} & -\frac{107901732302496541474451}{1034163905755060080000000} \\
35 & -\frac{3819454061700392316245227}{29232366402676364928000000} & -\frac{6977969679139244469850843}{43848549604014547392000000} \\
37 & -\frac{1017997603808580389531}{13167732613818182400000} & -\frac{12205851819174081932939}{197515989207272736000000} \\
39 & -\frac{60527320889180452507}{3242396539654272000000} & -\frac{53587459125338346701}{470743171894195200000} \\
41 & -\frac{3927082740146217147964921}{31975911667396897200000000} & -\frac{25527543258444304704547}{18546028767090200376000000} \\
44 & -\frac{495422665040944338081773}{10874586170098470980160000} & -\frac{12810460897135999418331473}{181243102834974516336000000} \\
46 & -\frac{458359624839252842120111}{3377200673943624528000000} & -\frac{50883777169267119518541469}{496448499069712805616000000} \\
48 & -\frac{222871123484691571740523}{2514936672085677840000000} & -\frac{86935861841649749775059}{685891819659730320000000} \\
50 & -\frac{8030069374901109057588721}{207978077248136741145600000} & -\frac{202119542775190218561569}{12234004544008043596800000} \\
52 & -\frac{218596887196724529253781}{1925722937482747603200000} & -\frac{484421384033660892656423}{12654750732029484249600000} \\
54 & -\frac{8005301923003477643747}{117726323633922088800000} & -\frac{220635236772538173197783}{4395116082333091315200000} \\
56 & -\frac{292122117631230918736681}{17849121837112341309600000} & -\frac{185662057596890569982120027}{3248540174354446118347200000}\\
  \bottomrule
\end{longtblr}

\bibliographystyle{amsplain}
\bibliography{LS-disk}

@article {MR2912745,
    AUTHOR = {Bonnaillie-No\"{e}l, Virginie},
     TITLE = {Harmonic oscillators with {N}eumann condition on the
              half-line},
   JOURNAL = {Commun. Pure Appl. Anal.},
  FJOURNAL = {Communications on Pure and Applied Analysis},
    VOLUME = {11},
      YEAR = {2012},
    NUMBER = {6},
     PAGES = {2221--2237},
      ISSN = {1534-0392,1553-5258},
   MRCLASS = {35P15 (35J10 65N30)},
  MRNUMBER = {2912745},
MRREVIEWER = {Michael\ A.\ Perelmuter},
       DOI = {10.3934/cpaa.2012.11.2221},
       URL = {https://doi.org/10.3934/cpaa.2012.11.2221},
}

@article {MR2231969,
    AUTHOR = {Fournais, S. and Helffer, B.},
     TITLE = {On the third critical field in {G}inzburg-{L}andau theory},
   JOURNAL = {Comm. Math. Phys.},
  FJOURNAL = {Communications in Mathematical Physics},
    VOLUME = {266},
      YEAR = {2006},
    NUMBER = {1},
     PAGES = {153--196},
      ISSN = {0010-3616,1432-0916},
   MRCLASS = {58E50 (35J50 35Q55 82D55)},
  MRNUMBER = {2231969},
MRREVIEWER = {Xingbin\ Pan},
       DOI = {10.1007/s00220-006-0006-4},
       URL = {https://doi.org/10.1007/s00220-006-0006-4},
}

@article {MR2394546,
    AUTHOR = {Fournais, Soeren and Helffer, Bernard},
     TITLE = {Strong diamagnetism for general domains and application},
      NOTE = {Festival Yves Colin de Verdi\`ere},
   JOURNAL = {Ann. Inst. Fourier (Grenoble)},
  FJOURNAL = {Universit\'{e} de Grenoble. Annales de l'Institut Fourier},
    VOLUME = {57},
      YEAR = {2007},
    NUMBER = {7},
     PAGES = {2389--2400},
      ISSN = {0373-0956,1777-5310},
   MRCLASS = {35J55 (35P15 82D55)},
  MRNUMBER = {2394546},
MRREVIEWER = {Zuchi\ Chen},
       DOI = {10.5802/aif.2337},
       URL = {https://doi.org/10.5802/aif.2337},
}

@book {MR2662319,
    AUTHOR = {Fournais, S{\o}ren and Helffer, Bernard},
     TITLE = {Spectral methods in surface superconductivity},
    SERIES = {Progress in Nonlinear Differential Equations and their
              Applications},
    VOLUME = {77},
 PUBLISHER = {Birkh\"{a}user Boston, Inc., Boston, MA},
      YEAR = {2010},
     PAGES = {xx+324},
      ISBN = {978-0-8176-4796-4},
   MRCLASS = {35-02 (35P15 35Q56 47F05 47N50 49N60 82D55)},
  MRNUMBER = {2662319},
MRREVIEWER = {Yuri\ A.\ Kordyukov},
}

@article {MR3324161,
    AUTHOR = {Fournais, S{\o}ren and Sundqvist, Mikael Persson},
     TITLE = {Lack of diamagnetism and the {L}ittle-{P}arks effect},
   JOURNAL = {Comm. Math. Phys.},
  FJOURNAL = {Communications in Mathematical Physics},
    VOLUME = {337},
      YEAR = {2015},
    NUMBER = {1},
     PAGES = {191--224},
      ISSN = {0010-3616,1432-0916},
   MRCLASS = {78A25},
  MRNUMBER = {3324161},
MRREVIEWER = {Salvatore\ Esposito},
       DOI = {10.1007/s00220-014-2267-7},
       URL = {https://doi.org/10.1007/s00220-014-2267-7},
}

@article {MR4947381,
    AUTHOR = {Helffer, Bernard and L\'{e}na, Corentin},
     TITLE = {Eigenvalues of the {N}eumann magnetic {L}aplacian in the unit
              disk},
   JOURNAL = {J. Math. Phys.},
  FJOURNAL = {Journal of Mathematical Physics},
    VOLUME = {66},
      YEAR = {2025},
    NUMBER = {8},
     PAGES = {Paper No. 081513, 23},
      ISSN = {0022-2488,1089-7658},
   MRCLASS = {35P15 (81Q10)},
  MRNUMBER = {4947381},
       DOI = {10.1063/5.0260068},
       URL = {https://doi.org/10.1063/5.0260068},
}

@misc{DLMF,
         key = {DLMF},
       title = "{\it NIST Digital Library of Mathematical Functions}",
howpublished = "\url{https://dlmf.nist.gov/}, Release 1.2.7 of 2026-06-15",
         url = "https://dlmf.nist.gov/",
        note = "F.~W.~J. Olver, A.~B. {Olde Daalhuis}, D.~W. Lozier, B.~I. Schneider,
                R.~F. Boisvert, C.~W. Clark, B.~R. Miller, B.~V. Saunders,
                H.~S. Cohl, and M.~A. McClain, eds."}

@book {MR407617,
    AUTHOR = {Kato, Tosio},
     TITLE = {Perturbation theory for linear operators},
    SERIES = {Grundlehren der Mathematischen Wissenschaften, Band 132},
   EDITION = {Second},
 PUBLISHER = {Springer-Verlag, Berlin-New York},
      YEAR = {1976},
     PAGES = {xxi+619},
   MRCLASS = {47-XX},
  MRNUMBER = {407617},
}

@misc{LenaSundqvistBound,
    AUTHOR = {L\'{e}na, Corentin and Sundqvist, Mikael},
     TITLE = {A magnetic eigenvalue bound in the disk},
      YEAR = {2026},
      NOTE = {arXiv:2605.24188 [math.SP]},
HOWPUBLISHED = {\url{https://arxiv.org/abs/2605.24188}},
}

@article{SaintJames,
title = {{\'E}tude du champ critique {$H_{c_3}$} dans une g{\'e}om{\'e}trie cylindrique},
journal = {Physics Letters},
volume = {15},
number = {1},
pages = {13--15},
year = {1965},
issn = {0031-9163},
doi = {10.1016/0031-9163(65)91101-7},
author = {D. Saint-James},
}

\end{document}